\documentclass[twoside,11pt,reqno]{amsart}
\usepackage{amsmath,amsthm,amssymb,amstext,amsfonts,amscd}
\usepackage{graphicx}
\usepackage{multirow}

\newtheorem{theorem}{Theorem}

\newtheorem{corollary}{Corollary}
\newtheorem{definition}{Definition}

\newtheorem{remark}{Remark}

\numberwithin{equation}{section}
\input{tcilatex}

\begin{document}
\title[Some growth analysis of entire functions in the form of vector valued
Dirichlet series.....]{Some growth analysis of entire functions in the form
of vector valued Dirichlet series on the basis of their (p, q)-th relative
Ritt order and (p, q)-th relative Ritt type}
\author[Tanmay Biswas]{Tanmay Biswas}
\address{T. Biswas : Rajbari, Rabindrapalli, R. N. Tagore Road, P.O.-
Krishnagar, Dist-Nadia, PIN-\ 741101, West Bengal, India}
\email{tanmaybiswas\_math@rediffmail.com}

\begin{abstract}
{\small In this paper we wish to study some growth properties of entire
functions represented by a vector valued Dirichlet series on the basis of }$%
\left( p,q\right) ${\small -th relative Ritt order, }$\left( p,q\right) $%
{\small -th relative Ritt type and }$\left( p,q\right) ${\small -th relative
Ritt weak type where }$p\geq 0${\small \ and }$q\geq 0${\small .}
\end{abstract}

\maketitle

\section{\textbf{Introduction and Definitions}}

\qquad Suppose $f\left( s\right) $ be an entire function of the complex
variable $s=\sigma +it$ ($\sigma $ and $t$ are real variables) defined by
everywhere absolutely convergent \emph{vector valued Dirichlet series }%
briefly known as \emph{VVDS}%
\begin{equation}
f\left( s\right) =\text{ }\overset{\infty }{\underset{n=1}{\sum }}\text{ }%
a_{n}e^{s\lambda _{n}}  \label{x}
\end{equation}%
where $a_{n}$'s belong to a Banach space $\left( E,\left\Vert .\right\Vert
\right) $ and $\lambda _{n}$'s are non-negative real numbers such that $%
0<\lambda _{n}<\lambda _{n+1}\left( n\geq 1\right) ,\lambda _{n}\rightarrow
+\infty $ as $n\rightarrow +\infty $ and satisfy the conditions $\underset{%
n\rightarrow +\infty }{\overline{\lim }}\frac{\log n}{\lambda _{n}}%
=D<+\infty $ and $\underset{n\rightarrow +\infty }{\overline{\lim }}\frac{%
\log \left\Vert a_{n}\right\Vert }{\lambda _{n}}=-\infty ~.$ If $\sigma _{c}$
and $\sigma _{a}$ denote respectively the abscissa of convergence and
absolute convergence of $\left( \ref{x}\right) $, then in this case clearly $%
\sigma _{a}=\sigma _{c}=+\infty .$ The function $M_{f}\left( \sigma \right) $
known as \emph{maximum modulus} function corresponding to an entire function 
$f\left( s\right) $ defined by $\left( \ref{x}\right) $, is written as
follows%
\begin{equation*}
M_{f}\left( \sigma \right) =\text{ }\underset{-\infty <t<+\infty }{l.u.b.}%
\left\Vert f\left( \sigma +it\right) \right\Vert ~.
\end{equation*}

\qquad Now we state the following two notations which are frequently use in
our subsequent study:%
\begin{eqnarray*}
\log ^{[k]}x &=&\log \left( \log ^{[k-1]}x\right) \text{ for }k=1,2,3,\cdot
\cdot \cdot \text{ }; \\
\log ^{[0]}x &=&x,\log ^{[-1]}x=\exp x
\end{eqnarray*}%
and%
\begin{eqnarray*}
\exp ^{[k]}x &=&\exp \left( \exp ^{[k-1]}x\right) \text{ for }k=1,2,3,\cdot
\cdot \cdot \text{ ;} \\
\exp ^{[0]}x &=&x,\exp ^{[-1]}x=\log x~.
\end{eqnarray*}

\qquad Juneja, Nandan and Kapoor \cite{3} first introduced the concept of $%
(p,q)$\emph{-th order }and\emph{\ }$(p,q)$\emph{-th lower order }of an\emph{%
\ entire Dirichlet series} where $p\geq q+1\geq 1$. In the line of Juneja et
al. \cite{3}, one can define the $(p,q)$\emph{\ -th Ritt order }and $(p,q)$%
\emph{-th Ritt lower order} of an entire function $f$ represented by \emph{%
VVDS} in the following way:%
\begin{equation*}
\begin{array}{c}
\rho _{f}\left( p,q\right) \\ 
\lambda _{f}\left( p,q\right)%
\end{array}%
=\underset{\sigma \rightarrow +\infty }{\lim }%
\begin{array}{c}
\sup \\ 
\inf%
\end{array}%
\frac{\log ^{[p]}M_{f}(\sigma )}{\log ^{\left[ q\right] }\sigma }=\underset{%
r\rightarrow +\infty }{\lim }%
\begin{array}{c}
\sup \\ 
\inf%
\end{array}%
\frac{\log ^{[p]}\sigma }{\log ^{\left[ q\right] }M_{f}^{-1}(\sigma )}~.
\end{equation*}

where $p\geq q+1\geq 1$.

\qquad In this connection let us recall that if $0<\rho _{f}\left(
p,q\right) <\infty ,$ then the following properties hold 
\begin{equation*}
\rho _{f}\left( p-n,q\right) =\infty ~\text{for }n<p,~\rho _{f}\left(
p,q-n\right) =0~\text{for }n<q\text{, and}
\end{equation*}%
\begin{equation*}
\rho _{f}\left( p+n,q+n\right) =1~\text{for }n=1,2,...
\end{equation*}

\qquad Similarly for $0<\lambda _{f}\left( p,q\right) <\infty ,$ one can
easily verify that 
\begin{equation*}
\lambda _{f}\left( p-n,q\right) =\infty ~\text{for }n<p,~\lambda _{f}\left(
p,q-n\right) =0~\text{for }n<q\text{, and}
\end{equation*}%
\begin{equation*}
\lambda _{f}\left( p+n,q+n\right) =1~\text{for }n=1,2,....~.
\end{equation*}

\qquad Recalling that for any pair of integer numbers $m,n$ the Kroenecker
function is defined by $\delta _{m,n}=1$ for $m=n$ and $\delta _{m,n}=0$ for 
$m\neq n$, the aforementioned properties provide the following definition.

\begin{definition}
\label{d2} An entire function $f$ represented by \emph{VVDS }is said to have
index-pair $\left( 1,1\right) $ if $0<\rho _{f}\left( 1,1\right) <\infty $.
Otherwise, $f$ is said to have index-pair $\left( p,q\right) \neq \left(
1,1\right) $, $p\geq q+1\geq 1$, if $\delta _{p-q,0}<\rho _{f}\left(
p,q\right) <\infty $ and $\rho _{f}\left( p-1,q-1\right) \notin \mathbb{R}%
^{+}.$
\end{definition}

\begin{definition}
\label{d3.} An entire function $f$ represented by \emph{VVDS }is said to
have lower index-pair $\left( 1,1\right) $ if $0<\lambda _{f}\left(
1,1\right) <\infty $. Otherwise, $f$ is said to have lower index-pair $%
\left( p,q\right) \neq \left( 1,1\right) $, $p\geq q+1\geq 1$, if $\delta
_{p-q,0}<\lambda _{f}\left( p,q\right) <\infty $ and $\lambda _{f}\left(
p-1,q-1\right) \notin \mathbb{R}^{+}.$
\end{definition}

\qquad An entire function $f$ (represented by \emph{VVDS) }of \emph{%
index-pair }$(p,q)$ is said to be of \emph{regular }$\left( p,q\right) $%
\emph{-Ritt growth} if its $\left( p,q\right) $\emph{-th Ritt order}
coincides with its\emph{\ }$(p,q)$\emph{-th Ritt lower order}, otherwise $f$
is said to be of \emph{irregular }$\left( p,q\right) $\emph{-Ritt growth}.

\qquad Now to compare the relative growth of two entire functions
represented by \emph{VVDS }having same non zero finite $\left( p,q\right) $%
\emph{-th Ritt order}, one may introduce the definition of $\left(
p,q\right) $\emph{-th Ritt type }(resp.\emph{\ }$\left( p,q\right) $-th 
\emph{Ritt lower type}) in the following manner:

\begin{definition}
\label{d4.} The $\left( p,q\right) $\emph{-th Ritt type }and\emph{\ }$\left(
p,q\right) $-th \emph{Ritt lower type} respectively denoted by $\Delta
_{f}\left( p,q\right) $ and $\overline{\Delta }_{f}\left( p,q\right) $ of an
entire function $f$ represented by \emph{VVDS when }$0<$ $\rho _{f}\left(
p,q\right) $ $<+\infty $ are defined as follows:%
\begin{equation*}
\begin{array}{c}
\Delta _{f}\left( p,q\right) \\ 
\overline{\Delta }_{f}\left( p,q\right)%
\end{array}%
=\underset{\sigma \rightarrow +\infty }{\lim }%
\begin{array}{c}
\sup \\ 
\inf%
\end{array}%
\frac{\log ^{\left[ p-1\right] }M_{f}\left( \sigma \right) }{\left[ \log ^{%
\left[ q-1\right] }\sigma \right] ^{\rho _{f}\left( p,q\right) }}~.
\end{equation*}%
where $p\geq q+1\geq 1.$
\end{definition}

\qquad Analogously to determine the relative growth of two entire functions
represented by vector valued Dirichlet series having same non zero finite%
\emph{\ }$\left( p,q\right) $-th \emph{Ritt lower order,} one may introduce
the definition of\emph{\ }$\left( p,q\right) $- th \emph{Ritt weak type} in
the following way:

\begin{definition}
\label{d5} The $\left( p,q\right) $- th \emph{Ritt weak type} denoted by $%
\tau _{f}\left( p,q\right) $ of an entire function $f$ represented by \emph{%
VVDS} is defined as follows:%
\begin{equation*}
\tau _{f}\left( p,q\right) =\underset{\sigma \rightarrow +\infty }{%
\underline{\lim }}\frac{\log ^{\left[ p-1\right] }M_{f}\left( \sigma \right) 
}{\left[ \log ^{\left[ q-1\right] }\sigma \right] ^{\lambda _{f}\left(
p,q\right) }}~,~0<\text{ }\lambda _{f}\left( p,q\right) \text{ }<+\infty ~.
\end{equation*}

Also one may define the growth indicator $\overline{\tau }_{f}\left(
p,q\right) $ of an entire function $f$ represented by \emph{VVDS} in the
following manner :%
\begin{equation*}
\overline{\tau }_{f}\left( p,q\right) =\underset{\sigma \rightarrow +\infty }%
{\overline{\lim }}\frac{\log ^{\left[ p-1\right] }M_{f}\left( \sigma \right) 
}{\left[ \log ^{\left[ q-1\right] }\sigma \right] ^{\lambda _{f}\left(
p,q\right) }}~,~0<\text{ }\lambda _{f}\left( p,q\right) \text{ }<+\infty ,
\end{equation*}%
where $p\geq q+1\geq 1.$
\end{definition}

\qquad The above definitions are extended the \emph{generalized Ritt growth
indicators }of an entire function $f$ represented by \emph{VVDS} for each
integer $p\geq 2$ and $q=0.$ Also for $p=2$ and $q=0,$ the above definitions
reduces to the classical definitions of an entire function $f$ represented
by \emph{VVDS.}

\qquad G. S. Srivastava \cite{7} introduced the \emph{relative Ritt order}
between two entire functions represented by \emph{VVDS} to avoid comparing
growth just with $\exp \exp z\ \ $In the case of \emph{relative Ritt order,}
it therefore seems reasonable to define suitably the $(p,q)$\emph{- th
relative Ritt order} of entire function represented by \emph{VVDS}.
Recently, Datta and Biswas \cite{2} introduce the concept of $\left(
p,q\right) $\emph{-th relative Ritt order} $\rho _{g}^{\left( p,q\right)
}\left( f\right) $ of an entire function $f$ represented by \emph{VVDS} with
respect to another entire function $g$ which is also represented by \emph{%
VVDS}, in the light of index-pair which is as follows:

\begin{definition}
\label{d6}\cite{2} Let $f$ and $g$ be any two entire functions represented
by \emph{VVDS }with index-pair $\left( m,q\right) $ and $\left( m,p\right) ,$
respectively, where $p,q,m$ are positive integers such that $m\geq q+1\geq 1$
and $m\geq p+1\geq 1.$ Then the $(p,q)$\emph{- th relative Ritt order} and $%
\left( p,q\right) $\emph{- th relative Ritt lower order }of $f$ with respect
to $g$ are defined as%
\begin{equation*}
\begin{array}{c}
\rho _{g}^{\left( p,q\right) }\left( f\right) \\ 
\lambda _{g}^{\left( p,q\right) }\left( f\right)%
\end{array}%
=\underset{\sigma \rightarrow +\infty }{\lim }%
\begin{array}{c}
\sup \\ 
\inf%
\end{array}%
\frac{\log ^{\left[ p\right] }M_{g}^{-1}M_{f}\left( \sigma \right) }{\log ^{%
\left[ q\right] }\sigma }=\underset{r\rightarrow +\infty }{\lim }%
\begin{array}{c}
\sup \\ 
\inf%
\end{array}%
\frac{\log ^{[p]}M_{g}^{-1}\left( \sigma \right) }{\log ^{\left[ q\right]
}M_{f}^{-1}(\sigma )}~.
\end{equation*}
\end{definition}

\qquad In this connection, we intend to give a definition of relative
index-pair of an entire function\emph{\ }with respect to another entire
function (both of which\ represented by \emph{VVDS) }which is relevant in
the sequel :

\begin{definition}
\label{d7} Let $f$ and $g$ be any two entire functions both\ represented by 
\emph{VVDS} with index-pairs $\left( m,q\right) $ and $\left( m,p\right) $
respectively where $m\geq q+1\geq 1$ and $m\geq p+1\geq 1.$ Then the entire
function $f$ is said to have relative index-pair $\left( p,q\right) $ with
respect to another entire function $g$, if $b<\rho _{g}^{\left( p,q\right)
}\left( f\right) <\infty $ and $\rho _{g}^{\left( p-1,q-1\right) }\left(
f\right) $ is not a nonzero finite number, where $b=1$ if $p=q=m$ and $b=0$
for otherwise. Moreover if $0<\rho _{g}^{\left( p,q\right) }\left( f\right)
<\infty ,$ then%
\begin{equation*}
\rho _{g}^{\left( p-n,q\right) }\left( f\right) =\infty ~\text{for }%
n<p,~\rho _{g}^{\left( p,q-n\right) }\left( f\right) =0~\text{for }n<q\text{
and}
\end{equation*}%
\begin{equation*}
\rho _{g}^{\left( p+n,q+n\right) }\left( f\right) =1~\text{for }n=1,2,....~.
\end{equation*}%
Similarly for $0<\lambda _{g}^{\left( p,q\right) }\left( f\right) <\infty ,$
one can easily verify that%
\begin{equation*}
\lambda _{g}^{\left( p-n,q\right) }\left( f\right) =\infty ~\text{for }%
n<p,~\lambda _{g}^{\left( p,q-n\right) }\left( f\right) =0~\text{for }n<q%
\text{ and}
\end{equation*}%
\begin{equation*}
\lambda _{g}^{\left( p+n,q+n\right) }\left( f\right) =1~\text{for }%
n=1,2,....~.
\end{equation*}
\end{definition}

\qquad Further an entire function $f$ (represented by \emph{VVDS) }for which 
$(p,q)$\emph{-th relative Ritt order }and\emph{\ }$(p,q)$\emph{-th relative
Ritt lower order} with respect to another entire function $g$ (represented
by \emph{VVDS) }are the same is called a function of \emph{regular relative }%
$(p,q)$\emph{\ Ritt growth} with respect to $g$. Otherwise, $f$ is said to be%
\emph{\ irregular relative }$\left( p,q\right) $\emph{\ Ritt growth}.with
respect to $g$.

\qquad Now in order to compare the relative growth of two entire functions
represented by \emph{VVDS} having same non zero finite $(p,q)$\emph{-th
relative Ritt order} with respect to another entire function represented by 
\emph{VVDS, }one may introduce the concepts of $(p,q)$\emph{-th relative
Ritt-type }(resp.\emph{\ }$(p,q)$-th \emph{relative Ritt lower type}) which
are as follows:

\begin{definition}
\label{d8} Let $f$ and $g$ be any two entire functions represented by \emph{%
VVDS }with index-pair $\left( m,q\right) $ and $\left( m,p\right) ,$
respectively, where $p,q,m$ are positive integers such that $m\geq q+1\geq 1$
and $m\geq p+1\geq 1$ and $0<$ $\rho _{g}^{\left( p,q\right) }\left(
f\right) $ $<+\infty .$ Then the $(p,q)$\emph{- th relative Ritt type} and $%
(p,q)$\emph{- th relative Ritt lower type }of $f$ with respect to $g$ are
defined as%
\begin{equation*}
\begin{array}{c}
\Delta _{g}^{(p,q)}\left( f\right) \\ 
\overline{\Delta }_{g}^{(p,q)}\left( f\right)%
\end{array}%
=\underset{\sigma \rightarrow +\infty }{\lim }%
\begin{array}{c}
\sup \\ 
\inf%
\end{array}%
\frac{\log ^{\left[ p-1\right] }M_{g}^{-1}M_{f}\left( \sigma \right) }{\left[
\log ^{\left[ q-1\right] }\sigma \right] ^{\rho _{g}^{\left( p,q\right)
}\left( f\right) }}~.
\end{equation*}
\end{definition}

\qquad Analogously to determine the relative growth of two entire functions
represented by \emph{VVDS} having same non zero finite $p,q)$\emph{-th
relative Ritt lower order }with respect to another entire function
represented by \emph{VVDS}, one may introduce the definition of $(p,q)$\emph{%
- th relative Ritt weak type} in the following way:

\begin{definition}
\label{d9} Let $f$ and $g$ be any two entire functions represented by \emph{%
VVDS }with index-pair $\left( m,q\right) $ and $\left( m,p\right) ,$
respectively, where $p,q,m$ are positive integers such that $m\geq q+1\geq 1$
and $m\geq p+1\geq 1.$ Then $(p,q)$\emph{-th relative Ritt weak type}
denoted by $\tau _{g}^{(p,q)}\left( f\right) $ of an entire function $f$
with respect to another entire function $g$ is defined as follows:%
\begin{equation*}
\tau _{g}^{\left( p,q\right) }\left( f\right) =\underset{\sigma \rightarrow
+\infty }{\underline{\lim }}\frac{\log ^{\left[ p-1\right]
}M_{g}^{-1}M_{f}\left( \sigma \right) }{\left[ \log ^{\left[ q-1\right]
}\sigma \right] ^{\lambda _{g}^{\left( p,q\right) }\left( f\right) }}~,~0<%
\text{ }\lambda _{g}^{\left( p,q\right) }\left( f\right) \text{ }<+\infty ~.
\end{equation*}

Similarly the growth indicator $\overline{\tau }_{g}^{\left( p,q\right)
}\left( f\right) $ of an entire function $f$ with respect to another entire
function $g$ both represented by \emph{VVDS} in the following manner :%
\begin{equation*}
\overline{\tau }_{g}^{\left( p,q\right) }\left( f\right) =\underset{\sigma
\rightarrow +\infty }{\overline{\lim }}\frac{\log ^{\left[ p-1\right]
}M_{g}^{-1}M_{f}\left( \sigma \right) }{\left[ \log ^{\left[ q-1\right]
}\sigma \right] ^{\lambda _{g}^{\left( p,q\right) }\left( f\right) }}~,~0<%
\text{ }\lambda _{g}^{\left( p,q\right) }\left( f\right) \text{ }<+\infty ~.
\end{equation*}
\end{definition}

\qquad If $f$ and $g$ (both $f$ and $g$ are represented by \emph{VVDS) }have
got index-pair $\left( m,0\right) $ and $\left( m,l\right) ,$ respectively,
then Definition \ref{d6}, Definition \ref{d8} and Definition \ref{d9}
reduces to the definition of \emph{generalized relative Ritt growth
indicators.}such as \emph{generalized relative Ritt order }$\rho _{g}^{\left[
l\right] }\left( f\right) $\emph{, generalized relative Ritt type }$\Delta
_{g}^{\left[ l\right] }\left( f\right) $ etc. If the entire functions $f$
and $g$ (both $f$ and $g$ are represented by \emph{VVDS) }have the same
index-pair $\left( p,0\right) $ where $p$ is any positive integer, we get
the definitions of \emph{relative Ritt growth indicators }such as\emph{\
relative Ritt order }$\rho _{g}\left( f\right) $\emph{, relative Ritt type }$%
\Delta _{g}\left( f\right) $ etc introduced by Srivastava \cite{7} and Datta
et al. \cite{1}. Further if $g=\exp ^{\left[ m\right] }z,$ then Definition %
\ref{d6}, Definition \ref{d8} and Definition \ref{d9} reduces to the $\left(
m,q\right) $\emph{\ th Ritt growth indicators }of an entire function $f$
represented by \emph{VVDS.} Also for $g=\exp ^{\left[ m\right] }z,$ \emph{%
relative Ritt growth indicators} reduces to the definition of \emph{%
generalized Ritt growth indicators.}such as \emph{generalized Ritt order }$%
\rho _{g}^{\left[ m\right] }\left( f\right) $\emph{, generalized Ritt type }$%
\Delta _{g}^{\left[ m\right] }\left( f\right) $ etc. Moreover, if $f$ is an
entire function with index-pair $\left( 2,0\right) $ and $g=\exp ^{\left[ 2%
\right] }z$, then Definition \ref{d6}, Definition \ref{d8} and Definition %
\ref{d9} becomes the classical definitions of $f$ represented by \emph{VVDS.}

\qquad During the past decades, several authors $\left\{ \text{cf. \cite{1},%
\cite{x}, \cite{4},\cite{5},\cite{6},\cite{8},\cite{9}, \cite{10}}\right\} $
made closed investigations on the properties of entire Dirichlet series in
different directions using the \emph{growth indicator }such as \emph{Ritt
order}. In the present paper we wish to establish some basic properties of
entire functions represented by a \emph{VVDS} on the basis of $\left(
p,q\right) $-th relative Ritt order, $\left( p,q\right) $-th relative Ritt
type and $\left( p,q\right) $-th relative Ritt weak type where $p\geq 0$ and 
$q\geq 0$.

\section{\textbf{Main Results}}

{\normalsize \qquad In this section we state the main results of the paper.}

\begin{theorem}
\label{l1} Let $f$, $g$ and $h$ be any three entire functions represented by 
\emph{vector valued Dirichlet series }and $p\geq 0$, $q\geq 0$ and $m\geq 0$%
. If $\left( m,q\right) $\emph{-th relative Ritt order} (resp. $\left(
m,q\right) $\emph{-th relative Ritt lower order}) of $f$ with respect to $h$
and $\left( m,p\right) $\emph{-th relative Ritt order} (resp. $\left(
m,p\right) $\emph{-th relative Ritt lower order}) of $g$ with respect to $h$
are respectively denoted by $\rho _{h}^{\left( m,q\right) }\left( f\right) $ 
$\left( \text{resp. }\lambda _{h}^{\left( m,q\right) }\left( f\right)
\right) $ and $\rho _{h}^{\left( m,p\right) }\left( g\right) $ $\left( \text{%
resp. }\lambda _{h}^{\left( m,p\right) }\left( g\right) \right) $, then%
\begin{multline*}
\frac{\lambda _{h}^{\left( m,q\right) }\left( f\right) }{\rho _{h}^{\left(
m,p\right) }\left( g\right) }\leq \lambda _{g}^{\left( p,q\right) }\left(
f\right) \leq \min \left\{ \frac{\lambda _{h}^{\left( m,q\right) }\left(
f\right) }{\lambda _{h}^{\left( m,p\right) }\left( g\right) },\frac{\rho
_{h}^{\left( m,q\right) }\left( f\right) }{\rho _{h}^{\left( m,p\right)
}\left( g\right) }\right\} \\
\leq \max \left\{ \frac{\lambda _{h}^{\left( m,q\right) }\left( f\right) }{%
\lambda _{h}^{\left( m,p\right) }\left( g\right) },\frac{\rho _{h}^{\left(
m,q\right) }\left( f\right) }{\rho _{h}^{\left( m,p\right) }\left( g\right) }%
\right\} \leq \rho _{g}^{\left( p,q\right) }\left( f\right) \leq \frac{\rho
_{h}^{\left( m,q\right) }\left( f\right) }{\lambda _{h}^{\left( m,p\right)
}\left( g\right) }~.
\end{multline*}
\end{theorem}

\begin{proof}
From the definitions of $\rho _{g}^{\left( p,q\right) }\left( f\right) $ and 
$\lambda _{g}^{\left( p,q\right) }\left( f\right) $ we get that

\begin{equation}
\log \rho _{g}^{\left( p,q\right) }\left( f\right) =\underset{\sigma
\rightarrow +\infty }{\overline{\lim }}\left[ \log ^{\left[ p+1\right]
}M_{g}^{-1}\left( \sigma \right) -\log ^{\left[ q+1\right]
}M_{f}^{-1}(\sigma )\right] ,  \label{2x}
\end{equation}%
\begin{equation}
\log \lambda _{g}^{\left( p,q\right) }\left( f\right) =\underset{\sigma
\rightarrow +\infty }{\underline{\lim }}\left[ \log ^{\left[ p+1\right]
}M_{g}^{-1}\left( \sigma \right) -\log ^{\left[ q+1\right]
}M_{f}^{-1}(\sigma )\right] .  \label{3x}
\end{equation}

\qquad Now from the definitions of $\rho _{h}^{\left( m,q\right) }\left(
f\right) $ and $\lambda _{h}^{\left( m,q\right) }\left( f\right) ,$ it
follows that%
\begin{eqnarray}
\log \rho _{h}^{\left( m,q\right) }\left( f\right) &=&\underset{\sigma
\rightarrow +\infty }{\overline{\lim }}\left[ \log ^{[m+1]}M_{h}^{-1}\left(
\sigma \right) -\log ^{\left[ q+1\right] }M_{f}^{-1}(\sigma )\right] \text{,}
\label{4x} \\
\log \lambda _{h}^{\left( m,q\right) }\left( f\right) &=&\underset{\sigma
\rightarrow +\infty }{\underline{\lim }}\left[ \log ^{[m+1]}M_{h}^{-1}\left(
\sigma \right) -\log ^{\left[ q+1\right] }M_{f}^{-1}(\sigma )\right] ~.
\label{5x}
\end{eqnarray}

\qquad Similarly, from the definitions of $\rho _{h}^{\left( m,p\right)
}\left( g\right) $ and $\lambda _{h}^{\left( m,p\right) }\left( g\right) ,$
we obtain that%
\begin{eqnarray}
\log \rho _{h}^{\left( m,p\right) }\left( g\right) &=&\underset{\sigma
\rightarrow +\infty }{\overline{\lim }}\left[ \log ^{[m+1]}M_{h}^{-1}\left(
\sigma \right) -\log ^{\left[ p+1\right] }M_{g}^{-1}\left( \sigma \right) %
\right] \text{,}  \label{6x} \\
\log \lambda _{h}^{\left( m,p\right) }\left( g\right) &=&\underset{\sigma
\rightarrow +\infty }{\underline{\lim }}\left[ \log ^{[m+1]}M_{h}^{-1}\left(
\sigma \right) -\log ^{\left[ p+1\right] }M_{g}^{-1}\left( \sigma \right) %
\right] ~.  \label{7x}
\end{eqnarray}

\qquad Therefore from $\left( \ref{3x}\right) ,$ $\left( \ref{5x}\right) $
and $\left( \ref{6x}\right) $, we get that%
\begin{multline*}
\log \lambda _{g}^{\left( p,q\right) }\left( f\right) =\underset{\sigma
\rightarrow +\infty }{\underline{\lim }}\left[ \log ^{[m+1]}M_{h}^{-1}\left(
\sigma \right) -\log ^{\left[ q+1\right] }M_{f}^{-1}(\sigma )\right. \\
\left. -\left( \log ^{[m+1]}M_{h}^{-1}\left( \sigma \right) -\log ^{\left[
p+1\right] }M_{g}^{-1}\left( \sigma \right) \right) \right]
\end{multline*}%
\begin{multline*}
i.e.,~\log \lambda _{g}^{\left( p,q\right) }\left( f\right) \geq \left[ 
\underset{\sigma \rightarrow +\infty }{\underline{\lim }}\left( \log
^{[m+1]}M_{h}^{-1}\left( \sigma \right) -\log ^{\left[ q+1\right]
}M_{f}^{-1}(\sigma )\right) \right. \\
\left. -\underset{\sigma \rightarrow +\infty }{\overline{\lim }}\left( \log
^{[m+1]}M_{h}^{-1}\left( \sigma \right) -\log ^{\left[ p+1\right]
}M_{g}^{-1}\left( \sigma \right) \right) \right]
\end{multline*}%
\begin{equation}
i.e.,~\log \lambda _{g}^{\left( p,q\right) }\left( f\right) \geq \left( \log
\lambda _{h}^{\left( m,q\right) }\left( f\right) -\log \rho _{h}^{\left(
m,p\right) }\left( g\right) \right) ~.  \label{8x}
\end{equation}

\qquad Similarly, from $\left( \ref{2x}\right) ,$ $\left( \ref{4x}\right) $
and $\left( \ref{7x}\right) $, it follows that%
\begin{multline*}
\log \rho _{g}^{\left( p,q\right) }\left( f\right) =\underset{\sigma
\rightarrow +\infty }{\overline{\lim }}\left[ \log ^{[m+1]}M_{h}^{-1}\left(
\sigma \right) -\log ^{\left[ q+1\right] }M_{f}^{-1}(\sigma )\right. \\
\left. -\left( \log ^{[m+1]}M_{h}^{-1}\left( \sigma \right) -\log ^{\left[
p+1\right] }M_{g}^{-1}\left( \sigma \right) \right) \right]
\end{multline*}%
\begin{multline*}
i.e.,~\log \rho _{g}^{\left( p,q\right) }\left( f\right) \leq \left[ 
\underset{\sigma \rightarrow +\infty }{\overline{\lim }}\left( \log
^{[m+1]}M_{h}^{-1}\left( \sigma \right) -\log ^{\left[ q+1\right]
}M_{f}^{-1}(\sigma )\right) \right. \\
\left. -\underset{\sigma \rightarrow +\infty }{\underline{\lim }}\left( \log
^{[m+1]}M_{h}^{-1}\left( \sigma \right) -\log ^{\left[ p+1\right]
}M_{g}^{-1}\left( \sigma \right) \right) \right]
\end{multline*}%
\begin{equation}
i.e.,~\log \rho _{g}^{\left( p,q\right) }\left( f\right) \leq \left( \log
\rho _{h}^{\left( m,q\right) }\left( f\right) -\log \lambda _{h}^{\left(
m,p\right) }\left( g\right) \right) ~.  \label{9x}
\end{equation}

\qquad Again, in view of $\left( \ref{3x}\right) $ we obtain that%
\begin{multline*}
\log \lambda _{g}^{\left( p,q\right) }\left( f\right) =\underset{\sigma
\rightarrow \infty }{\underline{\lim }}\left[ \log ^{[m+1]}M_{h}^{-1}\left(
\sigma \right) -\log ^{\left[ q+1\right] }M_{f}^{-1}(\sigma )\right. \\
\left. -\left( \log ^{[m+1]}M_{h}^{-1}\left( \sigma \right) -\log ^{\left[
p+1\right] }M_{g}^{-1}\left( \sigma \right) \right) \right]
\end{multline*}

\qquad By taking $A=\left( \log ^{[m+1]}M_{h}^{-1}\left( \sigma \right)
-\log ^{\left[ q+1\right] }M_{f}^{-1}(\sigma )\right) $ and \linebreak $%
B=\left( \log ^{[m+1]}M_{h}^{-1}\left( \sigma \right) -\log ^{\left[ p+1%
\right] }M_{g}^{-1}\left( \sigma \right) \right) ,$ we get from above that%
\begin{equation*}
\log \lambda _{g}^{\left( p,q\right) }\left( f\right) \leq \min \left[ 
\underset{\sigma \rightarrow +\infty }{\underline{\lim }}A+\underset{\sigma
\rightarrow +\infty }{\overline{\lim }}-B,\underset{\sigma \rightarrow
+\infty }{\overline{\lim }}A+\underset{\sigma \rightarrow +\infty }{%
\underline{\lim }}-B\right]
\end{equation*}%
\begin{equation*}
i.e.,~\log \lambda _{g}^{\left( p,q\right) }\left( f\right) \leq \min \left[ 
\underset{\sigma \rightarrow \infty }{\underline{\lim }}A-\underset{\sigma
\rightarrow \infty }{\underline{\lim }}B,\underset{\sigma \rightarrow \infty 
}{\overline{\lim }}A-\underset{\sigma \rightarrow \infty }{\overline{\lim }}B%
\right] ~.
\end{equation*}

\qquad Therefore in view of $\left( \ref{4x}\right) ,$ $\left( \ref{5x}%
\right) ,$ $\left( \ref{6x}\right) $ and $\left( \ref{7x}\right) $ we get
from above that%
\begin{equation}
\log \lambda _{g}^{\left( p,q\right) }\left( f\right) \leq \min \left\{ \log
\lambda _{h}^{\left( m,q\right) }\left( f\right) -\log \lambda _{h}^{\left(
m,p\right) }\left( g\right) ,\log \rho _{h}^{\left( m,q\right) }\left(
f\right) -\log \rho _{h}^{\left( m,p\right) }\left( g\right) \right\} ~.
\label{10x}
\end{equation}

\qquad Further from $\left( \ref{2x}\right) $ it follows that%
\begin{multline*}
\log \rho _{g}^{\left( p,q\right) }\left( f\right) =\underset{\sigma
\rightarrow +\infty }{\overline{\lim }}\left[ \log ^{[m+1]}M_{h}^{-1}\left(
\sigma \right) -\log ^{\left[ q+1\right] }M_{f}^{-1}(\sigma )\right. \\
\left. -\left( \log ^{[m+1]}M_{h}^{-1}\left( \sigma \right) -\log ^{\left[
p+1\right] }M_{g}^{-1}\left( \sigma \right) \right) \right]
\end{multline*}

\qquad By taking $A=\left( \log ^{[m+1]}M_{h}^{-1}\left( \sigma \right)
-\log ^{\left[ q+1\right] }M_{f}^{-1}(\sigma )\right) $ and \linebreak $%
B=\left( \log ^{[m+1]}M_{h}^{-1}\left( \sigma \right) -\log ^{\left[ p+1%
\right] }M_{g}^{-1}\left( \sigma \right) \right) ,$ we obtain from above that%
\begin{equation*}
\log \rho _{g}^{\left( p,q\right) }\left( f\right) \geq \max \left[ \underset%
{\sigma \rightarrow +\infty }{\underline{\lim }}A+\underset{\sigma
\rightarrow +\infty }{\overline{\lim }}-B,\underset{\sigma \rightarrow
+\infty }{\overline{\lim }}A+\underset{\sigma \rightarrow +\infty }{%
\underline{\lim }}-B\right]
\end{equation*}%
\begin{equation*}
i.e.,~\log \rho _{g}^{\left( p,q\right) }\left( f\right) \geq \max \left[ 
\underset{\sigma \rightarrow +\infty }{\underline{\lim }}A-\underset{\sigma
\rightarrow +\infty }{\underline{\lim }}B,\underset{\sigma \rightarrow
+\infty }{\overline{\lim }}A-\underset{\sigma \rightarrow +\infty }{%
\overline{\lim }}B\right] ~.
\end{equation*}

\qquad Therefore in view of $\left( \ref{4x}\right) ,$ $\left( \ref{5x}%
\right) ,$ $\left( \ref{6x}\right) $ and $\left( \ref{7x}\right) ,$ it
follows from above that%
\begin{equation}
\log \rho _{g}^{\left( p,q\right) }\left( f\right) \geq \max \left\{ \log
\lambda _{h}^{\left( m,q\right) }\left( f\right) -\log \lambda _{h}^{\left(
m,p\right) }\left( g\right) ,\log \rho _{h}^{\left( m,q\right) }\left(
f\right) -\log \rho _{h}^{\left( m,p\right) }\left( g\right) \right\} ~.
\label{11x}
\end{equation}

\qquad Thus the theorem follows from $\left( \ref{8x}\right) ,$ $\left( \ref%
{9x}\right) ,$ $\left( \ref{10x}\right) $ and $\left( \ref{11x}\right) ~.$
\end{proof}

\qquad In view of Theorem \ref{l1}, one can easily verify the following
corollaries:

\begin{corollary}
\label{C1} Let $f$, $g$ and $h$ be any three entire functions represented by 
\emph{vector valued Dirichlet series. }Also let $f$ be an entire function
with \emph{regular relative }$(m,q)$\emph{-Ritt growth} with respect to
entire function $h$ and $g$ be entire having relative index-pair $\left(
m,p\right) $ with respect to another entire function $h$ where $p\geq 0$, $%
q\geq 0$ and $m\geq 0$. Then%
\begin{equation*}
\lambda _{g}^{\left( p,q\right) }\left( f\right) =\frac{\rho _{h}^{\left(
m,q\right) }\left( f\right) }{\rho _{h}^{\left( m,p\right) }\left( g\right) }%
\text{\ \ \ and \ \ }\rho _{g}^{\left( p,q\right) }\left( f\right) =\frac{%
\rho _{h}^{\left( m,q\right) }\left( f\right) }{\lambda _{h}^{\left(
m,p\right) }\left( g\right) }~.
\end{equation*}%
In addition, if $\rho _{h}^{\left( m,q\right) }\left( f\right) =\rho
_{h}^{\left( m,p\right) }\left( g\right) ,$ then%
\begin{equation*}
\lambda _{g}^{\left( p,q\right) }\left( f\right) =\rho _{f}^{\left(
q,p\right) }\left( g\right) =1~.
\end{equation*}
\end{corollary}

\begin{corollary}
\label{C2} Let $f$, $g$ and $h$ be any three entire functions represented by 
\emph{vector valued Dirichlet series. }Also let $f$ be an entire function
with relative index-pair $\left( m,q\right) $ with respect to entire
function $h$ and $g$ be entire of \emph{regular relative }$(m,p)$\emph{-Ritt
growth} with respect to another entire function $h$ where $p\geq 0$, $q\geq
0 $ and $m\geq 0$. Then%
\begin{equation*}
\lambda _{g}^{\left( p,q\right) }\left( f\right) =\frac{\lambda _{h}^{\left(
m,q\right) }\left( f\right) }{\rho _{h}^{\left( m,p\right) }\left( g\right) }%
\text{ \ \ and \ \ }\rho _{g}^{\left( p,q\right) }\left( f\right) =\frac{%
\rho _{h}^{\left( m,q\right) }\left( f\right) }{\rho _{h}^{\left( m,p\right)
}\left( g\right) }~.
\end{equation*}%
In addition, if $\rho _{h}^{\left( m,q\right) }\left( f\right) =\rho
_{h}^{\left( m,p\right) }\left( g\right) ,$ then%
\begin{equation*}
\rho _{g}^{\left( p,q\right) }\left( f\right) =\lambda _{f}^{\left(
q,p\right) }\left( g\right) =1~.
\end{equation*}
\end{corollary}

\begin{corollary}
\label{C3} Let $f$, $g$ and $h$ be any three entire functions represented by 
\emph{vector valued Dirichlet series. }Also let $f$ and $g$ be any two
entire functions with \emph{regular relative }$(m,q)$\emph{-Ritt growth} and 
\emph{regular relative }$(m,p)$\emph{-th Ritt growth} with respect to entire
function $h$ respectively where $p\geq 0$, $q\geq 0$ and $m\geq 0$. Then%
\begin{equation*}
\lambda _{g}^{\left( p,q\right) }\left( f\right) =\rho _{g}^{\left(
p,q\right) }\left( f\right) =\frac{\rho _{h}^{\left( m,q\right) }\left(
f\right) }{\rho _{h}^{\left( m,p\right) }\left( g\right) }~.
\end{equation*}
\end{corollary}

\begin{corollary}
\label{C4} Let $f$, $g$ and $h$ be any three entire functions represented by 
\emph{vector valued Dirichlet series. }Also let $f$ and $g$ be any two
entire functions with r\emph{regular relative }$(m,q)$\emph{- Ritt growth}
and \emph{regular relative }$(m,p)$\emph{- Ritt growth} with respect to
entire function $h$ respectively where $p\geq 0$, $q\geq 0$ and $m\geq 0.$
Also suppose that $\rho _{h}^{\left( m,q\right) }\left( f\right) =\rho
_{h}^{\left( m,p\right) }\left( g\right) .$ Then%
\begin{equation*}
\lambda _{g}^{\left( p,q\right) }\left( f\right) =\rho _{g}^{\left(
p,q\right) }\left( f\right) =\lambda _{f}^{\left( q,p\right) }\left(
g\right) =\rho _{f}^{\left( q,p\right) }\left( g\right) =1~.
\end{equation*}
\end{corollary}

\begin{corollary}
\label{c5} Let $f$, $g$ and $h$ be any three entire functions represented by 
\emph{vector valued Dirichlet series. }Also let $f$ and $g$ be any two
entire functions with relative index-pairs $\left( m,q\right) $ and $\left(
m,p\right) $ with respect to entire function $h$ respectively where $p\geq 0$%
, $q\geq 0$ and $m\geq 0$ and either $f$ is not of regular relative $\left(
m,q\right) $ - Ritt growth or $g$ is not of regular relative $\left(
m,p\right) $ - Ritt growth, then%
\begin{equation*}
\rho _{g}^{\left( p,q\right) }\left( f\right) .\rho _{f}^{\left( q,p\right)
}\left( g\right) \geq 1~.
\end{equation*}%
If $f$ and $g$ are both of regular relative $\left( m,q\right) $- Ritt
growth and regular relative $\left( m,p\right) $ - Ritt growth with respect
to entire function $h$ respectively, then%
\begin{equation*}
\rho _{g}^{\left( p,q\right) }\left( f\right) .\rho _{f}^{\left( q,p\right)
}\left( g\right) =1~.
\end{equation*}
\end{corollary}

\begin{corollary}
\label{c6} Let $f$, $g$ and $h$ be any three entire functions represented by 
\emph{vector valued Dirichlet series. }Also let $f$ and $g$ be any two
entire functions with relative index-pairs $\left( m,q\right) $ and $\left(
m,p\right) $ with respect to entire function $h$ respectively where $p\geq 0$%
, $q\geq 0$ and $m\geq 0$ and either $f$ is not of regular relative $\left(
m,q\right) $ - Ritt growth or $g$ is not of regular relative $\left(
m,p\right) $ - Ritt growth, then%
\begin{equation*}
\lambda _{g}^{\left( p,q\right) }\left( f\right) .\lambda _{f}^{\left(
q,p\right) }\left( g\right) \leq 1~.
\end{equation*}%
If $f$ and $g$ are both of regular relative $\left( m,q\right) $ - Ritt
growth and regular relative $\left( m,p\right) $ -Ritt growth with respect
to entire function $h$ respectively, then%
\begin{equation*}
\lambda _{g}^{\left( p,q\right) }\left( f\right) .\lambda _{f}^{\left(
q,p\right) }\left( g\right) =1~.
\end{equation*}
\end{corollary}

\begin{corollary}
\label{c7} Let $f$, $g$ and $h$ be any three entire functions represented by 
\emph{vector valued Dirichlet series }and\emph{\ }$p\geq 0$, $q\geq 0$ and $%
m\geq 0.$Also let $f$ be an entire function with relative index-pair $\left(
m,q\right) $, Then%
\begin{eqnarray*}
\left( i\right) ~\lambda _{g}^{\left( p,q\right) }\left( f\right) &=&\infty ~%
\text{when }\rho _{h}^{\left( m,p\right) }\left( g\right) =0~, \\
\left( ii\right) ~\rho _{g}^{\left( p,q\right) }\left( f\right) &=&\infty ~%
\text{when }\lambda _{h}^{\left( m,p\right) }\left( g\right) =0~, \\
\left( iii\right) ~\lambda _{g}^{\left( p,q\right) }\left( f\right) &=&0~%
\text{when }\rho _{h}^{\left( m,p\right) }\left( g\right) =\infty
\end{eqnarray*}%
and%
\begin{equation*}
\left( iv\right) ~\rho _{g}^{\left( p,q\right) }\left( f\right) =0~\text{%
when }\lambda _{h}^{\left( m,p\right) }\left( g\right) =\infty ~.
\end{equation*}
\end{corollary}

\begin{corollary}
\label{c8} Let $f$, $g$ and $h$ be any three entire functions represented by 
\emph{vector valued Dirichlet series }and\emph{\ }$p\geq 0$, $q\geq 0$ and $%
m\geq 0.$Also let $g$ be an entire function with relative index-pair $\left(
m,p\right) $, Then%
\begin{eqnarray*}
\left( i\right) ~\rho _{g}^{\left( p,q\right) }\left( f\right) &=&0~\text{%
when }\rho _{h}^{\left( m,q\right) }\left( f\right) =0~, \\
\left( ii\right) ~\lambda _{g}^{\left( p,q\right) }\left( f\right) &=&0~%
\text{when }\lambda _{h}^{\left( m,q\right) }\left( f\right) =0~, \\
\left( iii\right) ~\rho _{g}^{\left( p,q\right) }\left( f\right) &=&\infty ~%
\text{when }\rho _{h}^{\left( m,q\right) }\left( f\right) =\infty
\end{eqnarray*}%
and%
\begin{equation*}
\left( iv\right) ~\lambda _{g}^{\left( p,q\right) }\left( f\right) =\infty ~%
\text{when }\lambda _{h}^{\left( m,q\right) }\left( f\right) =\infty ~.
\end{equation*}
\end{corollary}

\begin{remark}
\label{r1} Under the same conditions of Theorem \ref{l1}, one may write $%
\rho _{g}^{\left( p,q\right) }\left( f\right) =\frac{\rho _{h}^{\left(
m,q\right) }\left( f\right) }{\rho _{h}^{\left( m,p\right) }\left( g\right) }
$\ and $\lambda _{g}^{\left( p,q\right) }\left( f\right) =\frac{\lambda
_{h}^{\left( m,q\right) }\left( f\right) }{\lambda _{h}^{\left( m,p\right)
}\left( g\right) }$ when $\lambda _{h}^{\left( m,p\right) }\left( g\right)
=\rho _{h}^{\left( m,p\right) }\left( g\right) .$ Similarly $\rho
_{g}^{\left( p,q\right) }\left( f\right) =\frac{\lambda _{h}^{\left(
m,q\right) }\left( f\right) }{\lambda _{h}^{\left( m,p\right) }\left(
g\right) }$ and\ $\lambda _{g}^{\left( p,q\right) }\left( f\right) =\frac{%
\rho _{h}^{\left( m,q\right) }\left( f\right) }{\rho _{h}^{\left( m,p\right)
}\left( g\right) }$ when $\lambda _{h}^{\left( m,q\right) }\left( f\right)
=\rho _{h}^{\left( m,q\right) }\left( f\right) .$
\end{remark}

\qquad Next we prove our theorem based on $\left( p,q\right) $-th relative
Ritt type and $\left( p,q\right) $-th relative Ritt weak type of entire
functions represented by \emph{VVDS.}

\begin{theorem}
{\normalsize \label{t1} }Let $f$ and $g$ be any two entire functions VVDS
defined by $\left( \ref{x}\right) $ with relative index-pairs $\left(
m,q\right) $ and $\left( m,p\right) $ with respect to another entire
function $h$ VVDS defined by $\left( \ref{x}\right) $ respectively where $%
p\geq 0$, $q\geq 0$ and $m\geq 0.$ Then{\normalsize 
\begin{equation*}
\max \left\{ \left[ \frac{\overline{\Delta }_{h}^{\left( m,q\right) }\left(
f\right) }{\tau _{h}^{\left( m,p\right) }\left( g\right) }\right] ^{\frac{1}{%
\lambda _{h}^{\left( m,p\right) }\left( g\right) }},\left[ \frac{\Delta
_{h}^{\left( m,q\right) }\left( f\right) }{\overline{\tau }_{h}^{\left(
m,p\right) }\left( g\right) }\right] ^{\frac{1}{\lambda _{h}^{\left(
m,p\right) }\left( g\right) }}\right\} \leq \Delta _{g}^{\left( p,q\right)
}\left( f\right) \leq \left[ \frac{\Delta _{h}^{\left( m,q\right) }\left(
f\right) }{\overline{\Delta }_{h}^{\left( m,p\right) }\left( g\right) }%
\right] ^{\frac{1}{\rho _{h}^{\left( m,p\right) }\left( g\right) }}~.
\end{equation*}%
}
\end{theorem}

\begin{proof}
{\normalsize From the definitions of }$\Delta _{h}^{\left( m,q\right)
}\left( f\right) ${\normalsize \ and }$\overline{\Delta }_{h}^{\left(
m,q\right) }\left( f\right) ${\normalsize $,$ we have for all sufficiently
large values of }${\normalsize \sigma }${\normalsize \ that%
\begin{eqnarray}
M_{f}\left( \sigma \right) &\leq &M_{h}\left[ \exp ^{\left[ m-1\right] }%
\left[ \left( \Delta _{h}^{\left( m,q\right) }\left( f\right) +\varepsilon
\right) \left[ \log ^{\left[ q-1\right] }\sigma \right] ^{\rho _{h}^{\left(
m,q\right) }\left( f\right) }\right] \right] ,  \label{1} \\
M_{f}\left( \sigma \right) &\geq &M_{h}\left[ \exp ^{\left[ m-1\right] }%
\left[ \left( \overline{\Delta }_{h}^{\left( m,q\right) }\left( f\right)
-\varepsilon \right) \left[ \log ^{\left[ q-1\right] }\sigma \right] ^{\rho
_{h}^{\left( m,q\right) }\left( f\right) }\right] \right]  \label{2}
\end{eqnarray}%
and also for a sequence of values of }${\normalsize \sigma }${\normalsize \
tending to infinity, we get that%
\begin{eqnarray}
M_{f}\left( \sigma \right) &\geq &M_{h}\left[ \exp ^{\left[ m-1\right] }%
\left[ \left( \Delta _{h}^{\left( m,q\right) }\left( f\right) -\varepsilon
\right) \left[ \log ^{\left[ q-1\right] }\sigma \right] ^{\rho _{h}^{\left(
m,q\right) }\left( f\right) }\right] \right] ,  \label{3} \\
M_{f}\left( \sigma \right) &\leq &M_{h}\left[ \exp ^{\left[ m-1\right] }%
\left[ \left( \overline{\Delta }_{h}^{\left( m,q\right) }\left( f\right)
+\varepsilon \right) \left[ \log ^{\left[ q-1\right] }\sigma \right] ^{\rho
_{h}^{\left( m,q\right) }\left( f\right) }\right] \right] ~.  \label{4}
\end{eqnarray}%
}

\qquad {\normalsize Similarly from the definitions of }${\normalsize \Delta }%
_{h}^{\left( m,p\right) }\left( g\right) ${\normalsize \ and $\overline{%
\Delta }_{h}^{\left( m,p\right) }\left( g\right) ,$ it follows for all
sufficiently large values of }${\normalsize \sigma }${\normalsize \ that%
\begin{eqnarray}
M_{g}\left( \sigma \right) &\leq &M_{h}\left[ \exp ^{\left[ m-1\right] }%
\left[ \left( \Delta _{h}^{\left( m,p\right) }\left( g\right) +\varepsilon
\right) \left[ \log ^{\left[ p-1\right] }\sigma \right] ^{\rho _{h}^{\left(
m,p\right) }\left( g\right) }\right] \right]  \notag \\
i.e.,~M_{h}\left( \sigma \right) &\geq &M_{g}\left[ \exp ^{\left[ p-1\right]
}\left( \frac{\log ^{\left[ m-1\right] }\sigma }{\left( \Delta _{h}^{\left(
m,p\right) }\left( g\right) +\varepsilon \right) }\right) ^{\frac{1}{\rho
_{h}^{\left( m,p\right) }\left( g\right) }}\right] ~\text{and}  \label{5}
\end{eqnarray}%
\begin{equation}
M_{h}\left( \sigma \right) \leq M_{g}\left[ \exp ^{\left[ p-1\right] }\left( 
\frac{\log ^{\left[ m-1\right] }\sigma }{\left( \overline{\Delta }%
_{h}^{\left( m,p\right) }\left( g\right) -\varepsilon \right) }\right) ^{%
\frac{1}{\rho _{h}^{\left( m,p\right) }\left( g\right) }}\right] ~.
\label{6}
\end{equation}%
}

\qquad {\normalsize Also for a sequence of values of }${\normalsize \sigma }$%
{\normalsize \ tending to infinity, we obtain that%
\begin{equation}
M_{h}\left( \sigma \right) \leq M_{g}\left[ \exp ^{\left[ p-1\right] }\left( 
\frac{\log ^{\left[ m-1\right] }\sigma }{\left( \Delta _{h}^{\left(
m,p\right) }\left( g\right) -\varepsilon \right) }\right) ^{\frac{1}{\rho
_{h}^{\left( m,p\right) }\left( g\right) }}\right] ~\text{\ and}  \label{7}
\end{equation}%
\begin{equation}
M_{h}\left( \sigma \right) \geq M_{g}\left[ \exp ^{\left[ p-1\right] }\left( 
\frac{\log ^{\left[ m-1\right] }\sigma }{\left( \overline{\Delta }%
_{h}^{\left( m,p\right) }\left( g\right) +\varepsilon \right) }\right) ^{%
\frac{1}{\rho _{h}^{\left( m,p\right) }\left( g\right) }}\right] ~.~\ \ \ \
\   \label{8}
\end{equation}%
}

\qquad {\normalsize From the definitions of $\overline{\tau }$}$_{h}^{\left(
m,q\right) }\left( f\right) ${\normalsize \ and $\tau $}$_{h}^{\left(
m,q\right) }\left( f\right) ${\normalsize , we have for all sufficiently
large values of }$\sigma ${\normalsize \ that%
\begin{eqnarray}
M_{f}\left( \sigma \right) &\leq &M_{h}\left[ \exp ^{\left[ m-1\right] }%
\left[ \left( \overline{\tau }_{h}^{\left( m,q\right) }\left( f\right)
+\varepsilon \right) \left[ \log ^{\left[ q-1\right] }\sigma \right]
^{\lambda _{h}^{\left( m,q\right) }\left( f\right) }\right] \right] ,
\label{9} \\
M_{f}\left( \sigma \right) &\geq &M_{h}\left[ \exp ^{\left[ m-1\right] }%
\left[ \left( \tau _{h}^{\left( m,q\right) }\left( f\right) -\varepsilon
\right) \left[ \log ^{\left[ q-1\right] }\sigma \right] ^{\lambda
_{h}^{\left( m,q\right) }\left( f\right) }\right] \right]  \label{10}
\end{eqnarray}%
and also for a sequence of values of }${\normalsize \sigma }${\normalsize \
tending to infinity, we get that%
\begin{eqnarray}
M_{f}\left( \sigma \right) &\geq &M_{h}\left[ \exp ^{\left[ m-1\right] }%
\left[ \left( \overline{\tau }_{h}^{\left( m,q\right) }\left( f\right)
-\varepsilon \right) \left[ \log ^{\left[ q-1\right] }\sigma \right]
^{\lambda _{h}^{\left( m,q\right) }\left( f\right) }\right] \right] ,
\label{11} \\
M_{f}\left( \sigma \right) &\leq &M_{h}\left[ \exp ^{\left[ m-1\right] }%
\left[ \left( \tau _{h}^{\left( m,q\right) }\left( f\right) +\varepsilon
\right) \left[ \log ^{\left[ q-1\right] }\sigma \right] ^{\lambda
_{h}^{\left( m,q\right) }\left( f\right) }\right] \right] ~.  \label{12}
\end{eqnarray}%
}

\qquad {\normalsize Similarly from the definitions of $\overline{\tau }$}$%
_{h}^{\left( m,p\right) }\left( g\right) ${\normalsize \ and $\tau
_{h}^{\left( m,p\right) }\left( g\right) ,$ it follows for all sufficiently
large values of }${\normalsize \sigma }${\normalsize \ that%
\begin{eqnarray}
M_{g}\left( \sigma \right) &\leq &M_{h}\left[ \exp ^{\left[ m-1\right] }%
\left[ \left( \overline{\tau }_{h}^{\left( m,p\right) }\left( g\right)
+\varepsilon \right) \left[ \log ^{\left[ p-1\right] }\sigma \right]
^{\lambda _{h}^{\left( m,p\right) }\left( g\right) }\right] \right]  \notag
\\
i.e.,~M_{h}\left( \sigma \right) &\geq &M_{g}\left[ \exp ^{\left[ p-1\right]
}\left( \frac{\log ^{\left[ m-1\right] }\sigma }{\left( \overline{\tau }%
_{h}^{\left( m,p\right) }\left( g\right) +\varepsilon \right) }\right) ^{%
\frac{1}{\lambda _{h}^{\left( m,p\right) }\left( g\right) }}\right] ~\text{%
and}  \label{13}
\end{eqnarray}%
\begin{equation}
M_{h}\left( \sigma \right) \leq M_{g}\left[ \exp ^{\left[ p-1\right] }\left( 
\frac{\log ^{\left[ m-1\right] }\sigma }{\left( \tau _{h}^{\left( m,p\right)
}\left( g\right) -\varepsilon \right) }\right) ^{\frac{1}{\lambda
_{h}^{\left( m,p\right) }\left( g\right) }}\right] ~.~\ \ \ \ \ \ \ \ \ \ \
\   \label{14}
\end{equation}%
}

\qquad {\normalsize Also for a sequence of values of }${\normalsize \sigma }$%
{\normalsize \ tending to infinity, we obtain that%
\begin{equation}
M_{h}\left( \sigma \right) \leq M_{g}\left[ \exp ^{\left[ p-1\right] }\left( 
\frac{\log ^{\left[ m-1\right] }\sigma }{\left( \overline{\tau }_{h}^{\left(
m,p\right) }\left( g\right) -\varepsilon \right) }\right) ^{\frac{1}{\lambda
_{h}^{\left( m,p\right) }\left( g\right) }}\right] ~\text{\ and}  \label{15}
\end{equation}%
\begin{equation}
M_{h}\left( \sigma \right) \geq M_{g}\left[ \exp ^{\left[ p-1\right] }\left( 
\frac{\log ^{\left[ m-1\right] }\sigma }{\left( \tau _{h}^{\left( m,p\right)
}\left( g\right) +\varepsilon \right) }\right) ^{\frac{1}{\lambda _{g}\left(
m,p\right) }}\right] ~.~\ \ \ \ \   \label{16}
\end{equation}%
}

\qquad {\normalsize Now from $\left( \ref{3}\right) $ and in view of $\left( %
\ref{13}\right) $, we get for a sequence of values of }${\normalsize \sigma }
${\normalsize \ tending to infinity that%
\begin{equation*}
M_{g}^{-1}M_{f}\left( \sigma \right) \geq M_{g}^{-1}M_{h}\left[ \exp ^{\left[
m-1\right] }\left[ \left( \Delta _{h}^{\left( m,q\right) }\left( f\right)
-\varepsilon \right) \left[ \log ^{\left[ q-1\right] }\sigma \right] ^{\rho
_{h}^{\left( m,q\right) }\left( f\right) }\right] \right]
\end{equation*}%
\begin{equation*}
i.e.,~M_{g}^{-1}M_{f}\left( \sigma \right) \geq ~\ \ \ \ \ \ \ \ \ \ \ \ \ \
\ \ \ \ \ \ \ \ \ \ \ \ \ \ \ \ \ \ \ \ \ \ \ \ \ \ \ \ \ \ \ \ \ \ \ \ \ \
\ \ \ \ \ \ \ \ \ \ \ \ \ \ \ \ \ \ \ \ \ \ \ \ \ \ \ \ \ \ \ \ \ \ \ \ \ \
\ \ \ \ \ \ \ \ \ \ \ \ \ \ 
\end{equation*}%
}%
\begin{equation*}
~\ \ \ \ \ \ \ \ \ \ \ \ \ \ \ \ \ \ \left[ \exp ^{\left[ p-1\right] }\left( 
\frac{\log ^{\left[ m-1\right] }\exp ^{\left[ m-1\right] }\left[ \left(
\Delta _{h}^{\left( m,q\right) }\left( f\right) -\varepsilon \right) \left[
\log ^{\left[ q-1\right] }\sigma \right] ^{\rho _{h}^{\left( m,q\right)
}\left( f\right) }\right] }{\left( \overline{\tau }_{h}^{\left( m,p\right)
}\left( g\right) +\varepsilon \right) }\right) ^{\frac{1}{\lambda
_{h}^{\left( m,p\right) }\left( g\right) }}\right]
\end{equation*}%
{\normalsize 
\begin{equation}
i.e.,~\log ^{\left[ p-1\right] }M_{g}^{-1}M_{f}\left( \sigma \right) \geq %
\left[ \frac{\left( \Delta _{h}^{\left( m,q\right) }\left( f\right)
-\varepsilon \right) }{\left( \overline{\tau }_{h}^{\left( m,p\right)
}\left( g\right) +\varepsilon \right) }\right] ^{\frac{1}{\lambda
_{h}^{\left( m,p\right) }\left( g\right) }}\cdot \left[ \log ^{\left[ q-1%
\right] }\sigma \right] ^{\frac{\rho _{h}^{\left( m,q\right) }\left(
f\right) }{\lambda _{h}^{\left( m,p\right) }\left( g\right) }}~.  \notag
\end{equation}%
}

\qquad {\normalsize Since in view of Theorem \ref{l1}, }$\frac{\rho
_{h}^{\left( m,q\right) }\left( f\right) }{\lambda _{h}^{\left( m,p\right)
}\left( g\right) }${\normalsize $\geq $}$\rho _{g}^{\left( p,q\right)
}\left( f\right) ${\normalsize \ and as $\varepsilon \left( >0\right) $ is
arbitrary, therefore it follows from above that%
\begin{align}
\overline{\underset{\sigma \rightarrow \infty }{\lim }}\frac{\log ^{\left[
p-1\right] }M_{g}^{-1}M_{f}\left( \sigma \right) }{\left[ \log ^{\left[ q-1%
\right] }\sigma \right] ^{^{\rho _{g}^{\left( p,q\right) }\left( f\right) }}}%
& \geq \left[ \frac{\Delta _{h}^{\left( m,q\right) }\left( f\right) }{%
\overline{\tau }_{h}^{\left( m,p\right) }\left( g\right) }\right] ^{\frac{1}{%
\lambda _{h}^{\left( m,p\right) }\left( g\right) }}  \notag \\
i.e.,~\Delta _{g}^{\left( p,q\right) }\left( f\right) & \geq \left[ \frac{%
\Delta _{h}^{\left( m,q\right) }\left( f\right) }{\overline{\tau }%
_{h}^{\left( m,p\right) }\left( g\right) }\right] ^{\frac{1}{\lambda
_{h}^{\left( m,p\right) }\left( g\right) }}~.  \label{17}
\end{align}%
}

\qquad {\normalsize Similarly from $\left( \ref{2}\right) $ and in view of $%
\left( \ref{16}\right) $, it follows for a sequence of values of }$%
{\normalsize \sigma }${\normalsize \ tending to infinity that%
\begin{equation*}
M_{g}^{-1}M_{f}\left( \sigma \right) \geq M_{g}^{-1}M_{h}\left[ \exp ^{\left[
m-1\right] }\left[ \left( \overline{\Delta }_{h}^{\left( m,q\right) }\left(
f\right) -\varepsilon \right) \left[ \log ^{\left[ q-1\right] }\sigma \right]
^{\rho _{h}^{\left( m,q\right) }\left( f\right) }\right] \right]
\end{equation*}%
}%
\begin{equation*}
i.e.,~M_{g}^{-1}M_{f}\left( \sigma \right) \geq ~\ \ \ \ \ \ \ \ \ \ \ \ \ \
\ \ \ \ \ \ \ \ \ \ \ \ \ \ \ \ \ \ \ \ \ \ \ \ \ \ \ \ \ \ \ \ \ \ \ \ \ \
\ \ \ \ \ \ \ \ \ \ \ \ \ \ \ \ \ \ \ \ \ \ \ \ \ \ \ \ \ \ \ \ \ \ \ \ \ \
\ \ \ \ \ \ \ \ \ \ \ \ \ \ 
\end{equation*}%
\begin{equation*}
~\ \ \ \ \ \ \ \ \ \ \ \ \ \ \ \ \ \ \left[ \exp ^{\left[ p-1\right] }\left( 
\frac{\log ^{\left[ m-1\right] }\exp ^{\left[ m-1\right] }\left[ \left( 
\overline{\Delta }_{h}^{\left( m,q\right) }\left( f\right) -\varepsilon
\right) \left[ \log ^{\left[ q-1\right] }\sigma \right] ^{\rho _{h}^{\left(
m,q\right) }\left( f\right) }\right] }{\left( \tau _{h}^{\left( m,p\right)
}\left( g\right) +\varepsilon \right) }\right) ^{\frac{1}{\lambda _{g}\left(
m,p\right) }}\right]
\end{equation*}%
{\normalsize 
\begin{equation}
i.e.,~\log ^{\left[ p-1\right] }M_{g}^{-1}M_{f}\left( \sigma \right) \geq %
\left[ \frac{\left( \overline{\Delta }_{h}^{\left( m,q\right) }\left(
f\right) -\varepsilon \right) }{\left( \tau _{h}^{\left( m,p\right) }\left(
g\right) +\varepsilon \right) }\right] ^{\frac{1}{\lambda _{h}^{\left(
m,p\right) }\left( g\right) }}\cdot \left[ \log ^{\left[ q-1\right] }\sigma %
\right] ^{\frac{\rho _{h}^{\left( m,q\right) }\left( f\right) }{\lambda
_{h}^{\left( m,p\right) }\left( g\right) }}~.  \notag
\end{equation}%
}

\qquad {\normalsize Since in view of Theorem \ref{l1}, it follows that }$%
\frac{\rho _{h}^{\left( m,q\right) }\left( f\right) }{\lambda _{h}^{\left(
m,p\right) }\left( g\right) }${\normalsize $\geq \rho _{g}^{\left(
p,q\right) }\left( f\right) .$ Also $\varepsilon \left( >0\right) $ is
arbitrary, so we get from above that%
\begin{align}
\overline{\underset{\sigma \rightarrow \infty }{\lim }}\frac{\log ^{\left[
p-1\right] }M_{g}^{-1}M_{f}\left( \sigma \right) }{\left[ \log ^{\left[ q-1%
\right] }\sigma \right] ^{^{\rho _{g}^{\left( p,q\right) }\left( f\right) }}}%
& \geq \left[ \frac{\overline{\Delta }_{h}^{\left( m,q\right) }\left(
f\right) }{\tau _{h}^{\left( m,p\right) }\left( g\right) }\right] ^{\frac{1}{%
\lambda _{h}^{\left( m,p\right) }\left( g\right) }}  \notag \\
i.e.,~\Delta _{g}^{\left( p,q\right) }\left( f\right) & \geq \left[ \frac{%
\overline{\Delta }_{h}^{\left( m,q\right) }\left( f\right) }{\tau
_{h}^{\left( m,p\right) }\left( g\right) }\right] ^{\frac{1}{\lambda
_{h}^{\left( m,p\right) }\left( g\right) }}~.  \label{18}
\end{align}%
}

\qquad {\normalsize Again in view of $\left( \ref{6}\right) $, we have from $%
\left( \ref{1}\right) $ for all sufficiently large values of }${\normalsize %
\sigma }${\normalsize \ that%
\begin{equation*}
M_{g}^{-1}M_{f}\left( \sigma \right) \leq M_{g}^{-1}M_{h}\left[ \exp ^{\left[
m-1\right] }\left[ \left( \Delta _{h}^{\left( m,q\right) }\left( f\right)
+\varepsilon \right) \left[ \log ^{\left[ q-1\right] }\sigma \right] ^{\rho
_{h}^{\left( m,q\right) }\left( f\right) }\right] \right]
\end{equation*}%
}%
\begin{equation*}
i.e.,~M_{g}^{-1}M_{f}\left( \sigma \right) \leq ~\ \ \ \ \ \ \ \ \ \ \ \ \ \
\ \ \ \ \ \ \ \ \ \ \ \ \ \ \ \ \ \ \ \ \ \ \ \ \ \ \ \ \ \ \ \ \ \ \ \ \ \
\ \ \ \ \ \ \ \ \ \ \ \ \ \ \ \ \ \ \ \ \ \ \ \ \ \ \ \ \ \ \ \ \ \ \ \ \ \
\ \ \ \ \ \ \ \ \ \ \ \ \ \ 
\end{equation*}%
\begin{equation*}
~\ \ \ \ \ \ \ \ \ \ \ \ \ \ \ \ \ \ \left[ \exp ^{\left[ p-1\right] }\left( 
\frac{\log ^{\left[ m-1\right] }\exp ^{\left[ m-1\right] }\left[ \left(
\Delta _{h}^{\left( m,q\right) }\left( f\right) +\varepsilon \right) \left[
\log ^{\left[ q-1\right] }\sigma \right] ^{\rho _{h}^{\left( m,q\right)
}\left( f\right) }\right] }{\left( \overline{\Delta }_{h}^{\left( m,p\right)
}\left( g\right) -\varepsilon \right) }\right) ^{\frac{1}{\rho _{h}^{\left(
m,p\right) }\left( g\right) }}\right]
\end{equation*}%
{\normalsize 
\begin{equation}
i.e.,~\log ^{\left[ p-1\right] }M_{g}^{-1}M_{f}\left( \sigma \right) \leq %
\left[ \frac{\left( \Delta _{h}^{\left( m,q\right) }\left( f\right)
+\varepsilon \right) }{\left( \overline{\Delta }_{h}^{\left( m,p\right)
}\left( g\right) -\varepsilon \right) }\right] ^{\frac{1}{\rho _{h}^{\left(
m,p\right) }\left( g\right) }}\cdot \left[ \log ^{\left[ q-1\right] }\sigma %
\right] ^{\frac{\rho _{h}^{\left( m,q\right) }\left( f\right) }{\rho
_{h}^{\left( m,p\right) }\left( g\right) }}~.  \label{20}
\end{equation}%
}

\qquad {\normalsize As in view of Theorem \ref{l1}, it follows that }$\frac{%
\rho _{h}^{\left( m,q\right) }\left( f\right) }{\rho _{h}^{\left( m,p\right)
}\left( g\right) }${\normalsize $\leq $}$\rho _{g}^{\left( p,q\right)
}\left( f\right) ${\normalsize \ Since $\varepsilon \left( >0\right) $ is
arbitrary, we get from $\left( \ref{20}\right) $ that%
\begin{align}
\overline{\underset{\sigma \rightarrow \infty }{\lim }}\frac{\log ^{\left[
p-1\right] }M_{g}^{-1}M_{f}\left( \sigma \right) }{\left[ \log ^{\left[ q-1%
\right] }\sigma \right] ^{^{\rho _{g}^{\left( p,q\right) }\left( f\right) }}}%
& \leq \left[ \frac{\Delta _{h}^{\left( m,q\right) }\left( f\right) }{%
\overline{\Delta }_{h}^{\left( m,p\right) }\left( g\right) }\right] ^{\frac{1%
}{\rho _{h}^{\left( m,p\right) }\left( g\right) }}  \notag \\
i.e.,~\Delta _{g}^{\left( p,q\right) }\left( f\right) & \leq \left[ \frac{%
\Delta _{h}^{\left( m,q\right) }\left( f\right) }{\overline{\Delta }%
_{h}^{\left( m,p\right) }\left( g\right) }\right] ^{\frac{1}{\rho
_{h}^{\left( m,p\right) }\left( g\right) }}~.  \label{21}
\end{align}%
}

\qquad {\normalsize Thus the theorem follows from $\left( \ref{17}\right) $, 
$\left( \ref{18}\right) $ and $\left( \ref{21}\right) $. }
\end{proof}

{\normalsize \qquad The conclusion of the following corollary can be carried
out from $\left( \ref{6}\right) $ and $\left( \ref{9}\right) $; $\left( \ref%
{9}\right) $ and $\left( \ref{14}\right) $ respectively after applying the
same technique of Theorem \ref{t1} and with the help of Theorem \ref{l1}.
Therefore its proof is omitted. }

\begin{corollary}
{\normalsize \label{c1} }Let $f$ and $g$ be any two entire functions VVDS
defined by $\left( \ref{x}\right) $ with relative index-pairs $\left(
m,q\right) $ and $\left( m,p\right) $ with respect to another entire
function $h$ VVDS defined by $\left( \ref{x}\right) $ respectively where $%
p\geq 0$, $q\geq 0$ and $m\geq 0.$ Then{\normalsize 
\begin{equation*}
\Delta _{g}^{\left( p,q\right) }\left( f\right) \leq \min \left\{ \left[ 
\frac{\overline{\tau }_{h}^{\left( m,q\right) }\left( f\right) }{\tau
_{h}^{\left( m,p\right) }\left( g\right) }\right] ^{\frac{1}{\lambda
_{h}^{\left( m,p\right) }\left( g\right) }},\left[ \frac{\overline{\tau }%
_{h}^{\left( m,q\right) }\left( f\right) }{\overline{\Delta }_{h}^{\left(
m,p\right) }\left( g\right) }\right] ^{\frac{1}{\rho _{h}^{\left( m,p\right)
}\left( g\right) }}\right\} ~.
\end{equation*}%
}
\end{corollary}

{\normalsize \qquad Similarly in the line of Theorem \ref{t1} and with the
help of Theorem \ref{l1}, one may easily carried out the following theorem
from pairwise inequalities numbers $\left( \ref{10}\right) $ and $\left( \ref%
{13}\right) ;$ $\left( \ref{7}\right) $ and $\left( \ref{9}\right) $; $%
\left( \ref{6}\right) $ and $\left( \ref{12}\right) $ respectively and
therefore its proofs is omitted: }

\begin{theorem}
{\normalsize \label{t4} }Let $f$ and $g$ be any two entire functions VVDS
defined by $\left( \ref{x}\right) $ with relative index-pairs $\left(
m,q\right) $ and $\left( m,p\right) $ with respect to another entire
function $h$ VVDS defined by $\left( \ref{x}\right) $ respectively where $%
p\geq 0$, $q\geq 0$ and $m\geq 0.$ Then{\normalsize 
\begin{equation*}
\left[ \frac{\tau _{h}^{\left( m,q\right) }\left( f\right) }{\overline{\tau }%
_{h}^{\left( m,p\right) }\left( g\right) }\right] ^{\frac{1}{\lambda
_{h}^{\left( m,p\right) }\left( g\right) }}\leq \tau _{g}^{\left( p,q\right)
}\left( f\right) \leq \min \left\{ \left[ \frac{\tau _{h}^{\left( m,q\right)
}\left( f\right) }{\overline{\Delta }_{h}^{\left( m,p\right) }\left(
g\right) }\right] ^{\frac{1}{\rho _{h}^{\left( m,p\right) }\left( g\right) }%
},\left[ \frac{\overline{\tau }_{h}^{\left( m,q\right) }\left( f\right) }{%
\Delta _{h}^{\left( m,p\right) }\left( g\right) }\right] ^{\frac{1}{\rho
_{h}^{\left( m,p\right) }\left( g\right) }}\right\} ~.
\end{equation*}%
}
\end{theorem}

\begin{corollary}
{\normalsize \label{c4} }Let $f$ and $g$ be any two entire functions VVDS
defined by $\left( \ref{x}\right) $ with relative index-pairs $\left(
m,q\right) $ and $\left( m,p\right) $ with respect to another entire
function $h$ VVDS defined by $\left( \ref{x}\right) $ respectively where $%
p\geq 0$, $q\geq 0$ and $m\geq 0.$ Then{\normalsize 
\begin{equation*}
\tau _{g}^{\left( p,q\right) }\left( f\right) \geq \max \left\{ \left[ \frac{%
\overline{\Delta }_{h}^{\left( m,q\right) }\left( f\right) }{\Delta
_{h}^{\left( m,p\right) }\left( g\right) }\right] ^{\frac{1}{\rho
_{h}^{\left( m,p\right) }\left( g\right) }},\left[ \frac{\overline{\Delta }%
_{h}^{\left( m,q\right) }\left( f\right) }{\overline{\tau }_{h}^{\left(
m,p\right) }\left( g\right) }\right] ^{\frac{1}{\lambda _{h}^{\left(
m,p\right) }\left( g\right) }}\right\} ~.
\end{equation*}%
}
\end{corollary}

{\normalsize \qquad With the help of Theorem \ref{l1}, the conclusion of the
above corollary can be carry out from $\left( \ref{2}\right) ,$ $\left( \ref%
{5}\right) $ and $\left( \ref{2}\right) ,\left( \ref{13}\right) $
respectively after applying the same technique of Theorem \ref{t1} and
therefore its proof is omitted. }

\begin{theorem}
{\normalsize \label{t2} }Let $f$ and $g$ be any two entire functions VVDS
defined by $\left( \ref{x}\right) $ with relative index-pairs $\left(
m,q\right) $ and $\left( m,p\right) $ with respect to another entire
function $h$ VVDS defined by $\left( \ref{x}\right) $ respectively where $%
p\geq 0$, $q\geq 0$ and $m\geq 0.$ Then{\normalsize 
\begin{equation*}
\left[ \frac{\overline{\Delta }_{h}^{\left( m,q\right) }\left( f\right) }{%
\overline{\tau }_{h}^{\left( m,p\right) }\left( g\right) }\right] ^{\frac{1}{%
\lambda _{h}^{\left( m,p\right) }\left( g\right) }}\leq \overline{\Delta }%
_{g}^{\left( p,q\right) }\left( f\right) \leq \min \left\{ \left[ \frac{%
\overline{\Delta }_{h}^{\left( m,q\right) }\left( f\right) }{\overline{%
\Delta }_{h}^{\left( m,p\right) }\left( g\right) }\right] ^{\frac{1}{\rho
_{h}^{\left( m,p\right) }\left( g\right) }},\left[ \frac{\Delta _{h}^{\left(
m,q\right) }\left( f\right) }{\Delta _{h}^{\left( m,p\right) }\left(
g\right) }\right] ^{\frac{1}{\rho _{h}^{\left( m,p\right) }\left( g\right) }%
}\right\} ~.
\end{equation*}%
}
\end{theorem}

\begin{proof}
{\normalsize From $\left( \ref{2}\right) $ and in view of $\left( \ref{13}%
\right) $, we get for all sufficiently large values of }${\normalsize \sigma 
}${\normalsize \ that%
\begin{equation*}
M_{g}^{-1}M_{f}\left( \sigma \right) \geq M_{g}^{-1}M_{h}\left[ \exp ^{\left[
m-1\right] }\left[ \left( \overline{\Delta }_{h}^{\left( m,q\right) }\left(
f\right) -\varepsilon \right) \left[ \log ^{\left[ q-1\right] }\sigma \right]
^{\rho _{h}^{\left( m,q\right) }\left( f\right) }\right] \right]
\end{equation*}%
}%
\begin{equation*}
i.e.,~M_{g}^{-1}M_{f}\left( \sigma \right) \geq ~\ \ \ \ \ \ \ \ \ \ \ \ \ \
\ \ \ \ \ \ \ \ \ \ \ \ \ \ \ \ \ \ \ \ \ \ \ \ \ \ \ \ \ \ \ \ \ \ \ \ \ \
\ \ \ \ \ \ \ \ \ \ \ \ \ \ \ \ \ \ \ \ \ \ \ \ \ \ \ \ \ \ \ \ \ \ \ \ \ \
\ \ \ \ \ \ \ \ \ \ \ \ \ \ 
\end{equation*}%
\begin{equation*}
~\ \ \ \ \ \ \ \ \ \ \ \ \ \ \ \ \ \ \left[ \exp ^{\left[ p-1\right] }\left( 
\frac{\log ^{\left[ m-1\right] }\exp ^{\left[ m-1\right] }\left[ \left( 
\overline{\Delta }_{h}^{\left( m,q\right) }\left( f\right) -\varepsilon
\right) \left[ \log ^{\left[ q-1\right] }\sigma \right] ^{\rho _{h}^{\left(
m,q\right) }\left( f\right) }\right] }{\left( \overline{\tau }_{h}^{\left(
m,p\right) }\left( g\right) +\varepsilon \right) }\right) ^{\frac{1}{\lambda
_{h}^{\left( m,p\right) }\left( g\right) }}\right]
\end{equation*}%
{\normalsize 
\begin{equation}
i.e.,~\log ^{\left[ p-1\right] }M_{g}^{-1}M_{f}\left( \sigma \right) \geq %
\left[ \frac{\left( \overline{\Delta }_{h}^{\left( m,q\right) }\left(
f\right) -\varepsilon \right) }{\left( \overline{\tau }_{h}^{\left(
m,p\right) }\left( g\right) +\varepsilon \right) }\right] ^{\frac{1}{\lambda
_{h}^{\left( m,p\right) }\left( g\right) }}\cdot \left[ \log ^{\left[ q-1%
\right] }\sigma \right] ^{\frac{\rho _{h}^{\left( m,q\right) }\left(
f\right) }{\lambda _{h}^{\left( m,p\right) }\left( g\right) }}~.  \notag
\end{equation}%
}

{\normalsize Now in view of Theorem \ref{l1}, it follows that }$\frac{\rho
_{h}^{\left( m,q\right) }\left( f\right) }{\lambda _{h}^{\left( m,p\right)
}\left( g\right) }${\normalsize $\geq \rho _{g}^{\left( p,q\right) }\left(
f\right) .$ Since $\varepsilon \left( >0\right) $ is arbitrary, we get from
above that%
\begin{align}
\underset{r\rightarrow \infty }{\underline{\lim }}\frac{\log ^{\left[ p-1%
\right] }M_{g}^{-1}M_{f}\left( \sigma \right) }{\left[ \log ^{\left[ q-1%
\right] }\sigma \right] ^{^{\rho _{g}^{\left( p,q\right) }\left( f\right) }}}%
& \geq \left[ \frac{\overline{\Delta }_{h}^{\left( m,q\right) }\left(
f\right) }{\overline{\tau }_{h}^{\left( m,p\right) }\left( g\right) }\right]
^{\frac{1}{\lambda _{h}^{\left( m,p\right) }\left( g\right) }}  \notag \\
i.e.,~\overline{\Delta }_{g}^{\left( p,q\right) }\left( f\right) & \geq %
\left[ \frac{\overline{\Delta }_{h}^{\left( m,q\right) }\left( f\right) }{%
\overline{\tau }_{h}^{\left( m,p\right) }\left( g\right) }\right] ^{\frac{1}{%
\lambda _{h}^{\left( m,p\right) }\left( g\right) }}~.  \label{23}
\end{align}%
}

\qquad {\normalsize Further in view of $\left( \ref{7}\right) ,$ we get from 
$\left( \ref{1}\right) $ for a sequence of values of }${\normalsize \sigma }$%
{\normalsize \ tending to infinity that%
\begin{equation*}
M_{g}^{-1}M_{f}\left( \sigma \right) \leq M_{g}^{-1}M_{h}\left[ \exp ^{\left[
m-1\right] }\left[ \left( \Delta _{h}^{\left( m,q\right) }\left( f\right)
+\varepsilon \right) \left[ \log ^{\left[ q-1\right] }\sigma \right] ^{\rho
_{h}^{\left( m,q\right) }\left( f\right) }\right] \right]
\end{equation*}%
}%
\begin{equation*}
i.e.,~M_{g}^{-1}M_{f}\left( \sigma \right) \leq ~\ \ \ \ \ \ \ \ \ \ \ \ \ \
\ \ \ \ \ \ \ \ \ \ \ \ \ \ \ \ \ \ \ \ \ \ \ \ \ \ \ \ \ \ \ \ \ \ \ \ \ \
\ \ \ \ \ \ \ \ \ \ \ \ \ \ \ \ \ \ \ \ \ \ \ \ \ \ \ \ \ \ \ \ \ \ \ \ \ \
\ \ \ \ \ \ \ \ \ \ \ \ \ \ 
\end{equation*}%
\begin{equation*}
~\ \ \ \ \ \ \ \ \ \ \ \ \ \ \ \ \ \ \left[ \exp ^{\left[ p-1\right] }\left( 
\frac{\log ^{\left[ m-1\right] }\exp ^{\left[ m-1\right] }\left[ \left(
\Delta _{h}^{\left( m,q\right) }\left( f\right) +\varepsilon \right) \left[
\log ^{\left[ q-1\right] }\sigma \right] ^{\rho _{h}^{\left( m,q\right)
}\left( f\right) }\right] }{\left( \Delta _{h}^{\left( m,p\right) }\left(
g\right) -\varepsilon \right) }\right) ^{\frac{1}{\rho _{h}^{\left(
m,p\right) }\left( g\right) }}\right]
\end{equation*}%
{\normalsize 
\begin{equation}
i.e.,~\log ^{\left[ p-1\right] }M_{g}^{-1}M_{f}\left( \sigma \right) \leq %
\left[ \frac{\left( \Delta _{h}^{\left( m,q\right) }\left( f\right)
+\varepsilon \right) }{\left( \Delta _{h}^{\left( m,p\right) }\left(
g\right) -\varepsilon \right) }\right] ^{\frac{1}{\rho _{h}^{\left(
m,p\right) }\left( g\right) }}\cdot \left[ \log ^{\left[ q-1\right] }\sigma %
\right] ^{\frac{\rho _{h}^{\left( m,q\right) }\left( f\right) }{\rho
_{h}^{\left( m,p\right) }\left( g\right) }}~.  \label{24}
\end{equation}%
}

\qquad {\normalsize Again as in view of Theorem \ref{l1}, }$\frac{\rho
_{h}^{\left( m,q\right) }\left( f\right) }{\rho _{h}^{\left( m,p\right)
}\left( g\right) }${\normalsize $\leq $}$\rho _{g}^{\left( p,q\right)
}\left( f\right) ${\normalsize \ and $\varepsilon \left( >0\right) $ is
arbitrary, therefore we get from $\left( \ref{24}\right) $ that%
\begin{align}
\underset{r\rightarrow \infty }{\underline{\lim }}\frac{\log ^{\left[ p-1%
\right] }M_{g}^{-1}M_{f}\left( \sigma \right) }{\left[ \log ^{\left[ q-1%
\right] }\sigma \right] ^{^{\rho _{g}^{\left( p,q\right) }\left( f\right) }}}%
& \leq \left[ \frac{\Delta _{h}^{\left( m,q\right) }\left( f\right) }{\Delta
_{h}^{\left( m,p\right) }\left( g\right) }\right] ^{\frac{1}{\rho
_{h}^{\left( m,p\right) }\left( g\right) }}  \notag \\
i.e.,~\overline{\Delta }_{g}^{\left( p,q\right) }\left( f\right) & \leq %
\left[ \frac{\Delta _{h}^{\left( m,q\right) }\left( f\right) }{\Delta
_{h}^{\left( m,p\right) }\left( g\right) }\right] ^{\frac{1}{\rho
_{h}^{\left( m,p\right) }\left( g\right) }}~.  \label{25}
\end{align}%
}

\qquad {\normalsize Likewise from $\left( \ref{4}\right) $ and in view of $%
\left( \ref{6}\right) $, it follows for a sequence of values of }$%
{\normalsize \sigma }${\normalsize \ tending to infinity that%
\begin{equation*}
M_{g}^{-1}M_{f}\left( \sigma \right) \leq M_{g}^{-1}M_{h}\left[ \exp ^{\left[
m-1\right] }\left[ \left( \overline{\Delta }_{h}^{\left( m,q\right) }\left(
f\right) +\varepsilon \right) \left[ \log ^{\left[ q-1\right] }\sigma \right]
^{\rho _{h}^{\left( m,q\right) }\left( f\right) }\right] \right]
\end{equation*}%
}%
\begin{equation*}
i.e.,~M_{g}^{-1}M_{f}\left( \sigma \right) \leq ~\ \ \ \ \ \ \ \ \ \ \ \ \ \
\ \ \ \ \ \ \ \ \ \ \ \ \ \ \ \ \ \ \ \ \ \ \ \ \ \ \ \ \ \ \ \ \ \ \ \ \ \
\ \ \ \ \ \ \ \ \ \ \ \ \ \ \ \ \ \ \ \ \ \ \ \ \ \ \ \ \ \ \ \ \ \ \ \ \ \
\ \ \ \ \ \ \ \ \ \ \ \ \ \ 
\end{equation*}%
\begin{equation*}
~\ \ \ \ \ \ \ \ \ \ \ \ \ \ \ \ \ \ \left[ \exp ^{\left[ p-1\right] }\left( 
\frac{\log ^{\left[ m-1\right] }\exp ^{\left[ m-1\right] }\left[ \left( 
\overline{\Delta }_{h}^{\left( m,q\right) }\left( f\right) +\varepsilon
\right) \left[ \log ^{\left[ q-1\right] }\sigma \right] ^{\rho _{h}^{\left(
m,q\right) }\left( f\right) }\right] }{\left( \overline{\Delta }_{h}^{\left(
m,p\right) }\left( g\right) -\varepsilon \right) }\right) ^{\frac{1}{\rho
_{h}^{\left( m,p\right) }\left( g\right) }}\right]
\end{equation*}%
{\normalsize 
\begin{equation}
i.e.,~\log ^{\left[ p-1\right] }M_{g}^{-1}M_{f}\left( \sigma \right) \leq %
\left[ \frac{\left( \overline{\Delta }_{h}^{\left( m,q\right) }\left(
f\right) +\varepsilon \right) }{\left( \overline{\Delta }_{h}^{\left(
m,p\right) }\left( g\right) -\varepsilon \right) }\right] ^{\frac{1}{\rho
_{h}^{\left( m,p\right) }\left( g\right) }}\cdot \left[ \log ^{\left[ q-1%
\right] }\sigma \right] ^{\frac{\rho _{h}^{\left( m,q\right) }\left(
f\right) }{\rho _{h}^{\left( m,p\right) }\left( g\right) }}~.  \label{26}
\end{equation}%
}

\qquad {\normalsize Analogously, we get from $\left( \ref{26}\right) $ that%
\begin{align}
\underset{r\rightarrow \infty }{\underline{\lim }}\frac{\log ^{\left[ p-1%
\right] }M_{g}^{-1}M_{f}\left( \sigma \right) }{\left[ \log ^{\left[ q-1%
\right] }\sigma \right] ^{^{\rho _{g}^{\left( p,q\right) }\left( f\right) }}}%
& \leq \left[ \frac{\overline{\Delta }_{h}^{\left( m,q\right) }\left(
f\right) }{\overline{\Delta }_{h}^{\left( m,p\right) }\left( g\right) }%
\right] ^{\frac{1}{\rho _{h}^{\left( m,p\right) }\left( g\right) }}  \notag
\\
i.e.,~~\overline{\Delta }_{g}^{\left( p,q\right) }\left( f\right) & \leq %
\left[ \frac{\overline{\Delta }_{h}^{\left( m,q\right) }\left( f\right) }{%
\overline{\Delta }_{h}^{\left( m,p\right) }\left( g\right) }\right] ^{\frac{1%
}{\rho _{h}^{\left( m,p\right) }\left( g\right) }},  \label{27}
\end{align}%
since in view of Theorem \ref{l1}, }$\frac{\rho _{h}^{\left( m,q\right)
}\left( f\right) }{\rho _{h}^{\left( m,p\right) }\left( g\right) }$%
{\normalsize $\leq $}$\rho _{g}^{\left( p,q\right) }\left( f\right) $%
{\normalsize \ and $\varepsilon \left( >0\right) $ is arbitrary. }

\qquad {\normalsize Thus the theorem follows from $\left( \ref{23}\right) $, 
$\left( \ref{25}\right) $ and $\left( \ref{27}\right) $. }
\end{proof}

\begin{corollary}
{\normalsize \label{c2} }Let $f$ and $g$ be any two entire functions VVDS
defined by $\left( \ref{x}\right) $ with relative index-pairs $\left(
m,q\right) $ and $\left( m,p\right) $ with respect to another entire
function $h$ VVDS defined by $\left( \ref{x}\right) $ respectively where $%
p\geq 0$, $q\geq 0$ and $m\geq 0.$ Then{\normalsize 
\begin{equation*}
\overline{\Delta }_{g}^{\left( p,q\right) }\left( f\right) \leq ~\ \ \ \ \ \
\ \ \ \ \ \ \ \ \ \ \ \ \ \ \ \ \ \ \ \ \ \ \ \ \ \ \ \ \ \ \ \ \ \ \ \ \ \
\ \ \ \ \ \ \ \ \ \ \ \ \ \ \ \ \ \ \ \ \ \ \ \ \ \ \ \ \ \ \ \ \ \ \ \ \ \
\ \ \ \ \ \ \ \ \ \ \ \ \ \ \ \ \ \ \ \ \ \ \ \ \ \ \ \ \ \ \ \ \ \ \ \ \ \
\ 
\end{equation*}%
\begin{equation*}
\min \left\{ \left[ \frac{\tau _{h}^{\left( m,q\right) }\left( f\right) }{%
\tau _{h}^{\left( m,p\right) }\left( g\right) }\right] ^{\frac{1}{\lambda
_{h}^{\left( m,p\right) }\left( g\right) }},\left[ \frac{\overline{\tau }%
_{h}^{\left( m,q\right) }\left( f\right) }{\overline{\tau }_{h}^{\left(
m,p\right) }\left( g\right) }\right] ^{\frac{1}{\lambda _{h}^{\left(
m,p\right) }\left( g\right) }},\left[ \frac{\overline{\tau }_{h}^{\left(
m,q\right) }\left( f\right) }{\sigma _{h}^{\left( m,p\right) }\left(
g\right) }\right] ^{\frac{1}{\rho _{h}^{\left( m,p\right) }\left( g\right) }%
},\left[ \frac{\tau _{h}^{\left( m,q\right) }\left( f\right) }{\overline{%
\sigma }_{h}^{\left( m,p\right) }\left( g\right) }\right] ^{\frac{1}{\rho
_{h}^{\left( m,p\right) }\left( g\right) }}\right\} ~.
\end{equation*}%
}
\end{corollary}

{\normalsize \qquad The conclusion of the above corollary can be carried out
from pairwise inequalities no $\left( \ref{6}\right) $ and $\left( \ref{12}%
\right) ;$ $\left( \ref{7}\right) $ and $\left( \ref{9}\right) ;$ $\left( %
\ref{12}\right) $ and $\left( \ref{14}\right) $; $\left( \ref{9}\right) $
and $\left( \ref{15}\right) $ respectively after applying the same technique
of Theorem \ref{t2} and with the help of Theorem \ref{l1}. Therefore its
proof is omitted. }

{\normalsize \qquad Similarly in the line of Theorem \ref{t1} and with the
help of Theorem \ref{l1}, one may easily carried out the following theorem
from pairwise inequalities no $\left( \ref{11}\right) $ and $\left( \ref{13}%
\right) ;$ $\left( \ref{10}\right) $ and $\left( \ref{16}\right) $; $\left( %
\ref{6}\right) $ and $\left( \ref{9}\right) $ respectively and therefore its
proofs is omitted: }

\begin{theorem}
{\normalsize \label{t3} }Let $f$ and $g$ be any two entire functions VVDS
defined by $\left( \ref{x}\right) $ with relative index-pairs $\left(
m,q\right) $ and $\left( m,p\right) $ with respect to another entire
function $h$ VVDS defined by $\left( \ref{x}\right) $ respectively where $%
p\geq 0$, $q\geq 0$ and $m\geq 0.$ Then{\normalsize 
\begin{equation*}
\max \left\{ \left[ \frac{\overline{\tau }_{h}^{\left( m,q\right) }\left(
f\right) }{\overline{\tau }_{h}^{\left( m,p\right) }\left( g\right) }\right]
^{\frac{1}{\lambda _{g}\left( m,p\right) }},\left[ \frac{\tau _{h}^{\left(
m,q\right) }\left( f\right) }{\tau _{h}^{\left( m,p\right) }\left( g\right) }%
\right] ^{\frac{1}{\lambda _{h}^{\left( m,p\right) }\left( g\right) }%
}\right\} \leq \overline{\tau }_{g}^{\left( p,q\right) }\left( f\right) \leq %
\left[ \frac{\overline{\tau }_{h}^{\left( m,q\right) }\left( f\right) }{%
\overline{\Delta }_{h}^{\left( m,p\right) }\left( g\right) }\right] ^{\frac{1%
}{\rho _{h}^{\left( m,p\right) }\left( g\right) }}~.
\end{equation*}%
}
\end{theorem}

\begin{corollary}
{\normalsize \label{c3} }Let $f$ and $g$ be any two entire functions VVDS
defined by $\left( \ref{x}\right) $ with relative index-pairs $\left(
m,q\right) $ and $\left( m,p\right) $ with respect to another entire
function $h$ VVDS defined by $\left( \ref{x}\right) $ respectively where $%
p\geq 0$, $q\geq 0$ and $m\geq 0.$ Then{\normalsize 
\begin{equation*}
\overline{\tau }_{g}^{\left( p,q\right) }\left( f\right) \geq ~\ \ \ \ \ \ \
\ \ \ \ \ \ \ \ \ \ \ \ \ \ \ \ \ \ \ \ \ \ \ \ \ \ \ \ \ \ \ \ \ \ \ \ \ \
\ \ \ \ \ \ \ \ \ \ \ \ \ \ \ \ \ \ \ \ \ \ \ \ \ \ \ \ \ \ \ \ \ \ \ \ \ \
\ \ \ \ \ \ \ \ \ \ \ \ \ \ \ \ \ \ \ \ \ \ \ \ \ \ \ \ \ \ \ \ \ \ \ \ \ \
\ \ 
\end{equation*}%
\begin{equation*}
\max \left\{ \left[ \frac{\overline{\Delta }_{h}^{\left( m,q\right) }\left(
f\right) }{\overline{\Delta }_{h}^{\left( m,p\right) }\left( g\right) }%
\right] ^{\frac{1}{\rho _{h}^{\left( m,p\right) }\left( g\right) }},\left[ 
\frac{\Delta _{h}^{\left( m,q\right) }\left( f\right) }{\Delta _{h}^{\left(
m,p\right) }\left( g\right) }\right] ^{\frac{1}{\rho _{h}^{\left( m,p\right)
}\left( g\right) }},\left[ \frac{\Delta _{h}^{\left( m,q\right) }\left(
f\right) }{\overline{\tau }_{h}^{\left( m,p\right) }\left( g\right) }\right]
^{\frac{1}{\lambda _{h}^{\left( m,p\right) }\left( g\right) }},\left[ \frac{%
\overline{\Delta }_{h}^{\left( m,q\right) }\left( f\right) }{\tau
_{h}^{\left( m,p\right) }\left( g\right) }\right] ^{\frac{1}{\lambda
_{h}^{\left( m,p\right) }\left( g\right) }}\right\} ~.
\end{equation*}%
}
\end{corollary}

{\normalsize \qquad The conclusion of the above corollary can be carried out
from pairwise inequalities no $\left( \ref{3}\right) $ and $\left( \ref{5}%
\right) ;$ $\left( \ref{2}\right) $ and $\left( \ref{8}\right) ;$ $\left( %
\ref{3}\right) $ and $\left( \ref{13}\right) $; $\left( \ref{2}\right) $ and 
$\left( \ref{16}\right) $ respectively after applying the same technique of
Theorem \ref{t2} and with the help of Theorem \ref{l1}. Therefore its proof
is omitted.}

\qquad Now we state the following two theorems without their proofs as
because they can be derived easily using the same technique or with some
easy reasoning by the help of with the help of Remark \ref{r1} and therefore
left to the readers.

\begin{theorem}
\label{t4.1} Let $f$ and $g$ be any two entire functions VVDS defined by $%
\left( \ref{x}\right) $ with relative index-pairs $\left( m,q\right) $ and $%
\left( m,p\right) $ with respect to another entire function $h$ VVDS defined
by $\left( \ref{x}\right) $ respectively where $p\geq 0$, $q\geq 0$ and $%
m\geq 0.$ Also let $\lambda _{h}^{\left( m,p\right) }\left( g\right) =\rho
_{h}^{\left( m,p\right) }\left( g\right) $. Then%
\begin{multline*}
\left[ \frac{\overline{\Delta }_{h}^{\left( m,q\right) }\left( f\right) }{%
\Delta _{h}^{\left( m,p\right) }\left( g\right) }\right] ^{\frac{1}{\rho
_{h}^{\left( m,p\right) }\left( g\right) }}\leq \overline{\Delta }%
_{g}^{\left( p,q\right) }\left( f\right) \leq \min \left\{ \left[ \frac{%
\overline{\Delta }_{h}^{\left( m,q\right) }\left( f\right) }{\overline{%
\Delta }_{h}^{\left( m,p\right) }\left( g\right) }\right] ^{\frac{1}{\rho
_{h}^{\left( m,p\right) }\left( g\right) }},\left[ \frac{\Delta _{h}^{\left(
m,q\right) }\left( f\right) }{\Delta _{h}^{\left( m,p\right) }\left(
g\right) }\right] ^{\frac{1}{\rho _{h}^{\left( m,p\right) }\left( g\right) }%
}\right\} \\
\leq \max \left\{ \left[ \frac{\overline{\Delta }_{h}^{\left( m,q\right)
}\left( f\right) }{\overline{\Delta }_{h}^{\left( m,p\right) }\left(
g\right) }\right] ^{\frac{1}{\rho _{h}^{\left( m,p\right) }\left( g\right) }%
},\left[ \frac{\Delta _{h}^{\left( m,q\right) }\left( f\right) }{\Delta
_{h}^{\left( m,p\right) }\left( g\right) }\right] ^{\frac{1}{\rho
_{h}^{\left( m,p\right) }\left( g\right) }}\right\} \leq \Delta _{g}^{\left(
p,q\right) }\left( f\right) \leq \left[ \frac{\Delta _{h}^{\left( m,q\right)
}\left( f\right) }{\overline{\Delta }_{h}^{\left( m,p\right) }\left(
g\right) }\right] ^{\frac{1}{\rho _{h}^{\left( m,p\right) }\left( g\right) }}
\end{multline*}%
and%
\begin{multline*}
\left[ \frac{\tau _{h}^{\left( m,q\right) }\left( f\right) }{\overline{\tau }%
_{h}^{\left( m,p\right) }\left( g\right) }\right] ^{\frac{1}{\lambda
_{h}^{\left( m,p\right) }\left( g\right) }}\leq \tau _{g}^{\left( p,q\right)
}\left( f\right) \leq \min \left\{ \left[ \frac{\tau _{h}^{\left( m,q\right)
}\left( f\right) }{\tau _{h}^{\left( m,p\right) }\left( g\right) }\right] ^{%
\frac{1}{\lambda _{g}\left( _{h}^{\left( m,p\right) }\left( g\right)
m,p\right) }},\left[ \frac{\overline{\tau }_{h}^{\left( m,q\right) }\left(
f\right) }{\overline{\tau }_{h}^{\left( m,p\right) }\left( g\right) }\right]
^{\frac{1}{\lambda _{h}^{\left( m,p\right) }\left( g\right) }}\right\} \\
\leq \max \left\{ \left[ \frac{\tau _{h}^{\left( m,q\right) }\left( f\right) 
}{\tau _{h}^{\left( m,p\right) }\left( g\right) }\right] ^{\frac{1}{\lambda
_{g}\left( _{h}^{\left( m,p\right) }\left( g\right) m,p\right) }},\left[ 
\frac{\overline{\tau }_{h}^{\left( m,q\right) }\left( f\right) }{\overline{%
\tau }_{h}^{\left( m,p\right) }\left( g\right) }\right] ^{\frac{1}{\lambda
_{h}^{\left( m,p\right) }\left( g\right) }}\right\} \leq \overline{\tau }%
_{g}^{\left( p,q\right) }\left( f\right) \leq \left[ \frac{\overline{\tau }%
_{h}^{\left( m,q\right) }\left( f\right) }{\tau _{h}^{\left( m,p\right)
}\left( g\right) }\right] ^{\frac{1}{\lambda _{h}^{\left( m,p\right) }\left(
g\right) }}~.
\end{multline*}
\end{theorem}

\begin{theorem}
\label{t4.2} Let $f$ and $g$ be any two entire functions VVDS defined by $%
\left( \ref{x}\right) $ with relative index-pairs $\left( m,q\right) $ and $%
\left( m,p\right) $ with respect to another entire function $h$ VVDS defined
by $\left( \ref{x}\right) $ respectively where $p\geq 0$, $q\geq 0$ and $%
m\geq 0.$ Also let $\lambda _{h}^{\left( m,q\right) }\left( f\right) =\rho
_{h}^{\left( m,q\right) }\left( f\right) $. Then%
\begin{multline*}
\left[ \frac{\overline{\Delta }_{h}^{\left( m,q\right) }\left( f\right) }{%
\Delta _{h}^{\left( m,p\right) }\left( g\right) }\right] ^{\frac{1}{\rho
_{h}^{\left( m,p\right) }\left( g\right) }}\leq \tau _{g}^{\left( p,q\right)
}\left( f\right) \leq \min \left\{ \left[ \frac{\overline{\Delta }%
_{h}^{\left( m,q\right) }\left( f\right) }{\overline{\Delta }_{h}^{\left(
m,p\right) }\left( g\right) }\right] ^{\frac{1}{\rho _{h}^{\left( m,p\right)
}\left( g\right) }},\left[ \frac{\Delta _{h}^{\left( m,q\right) }\left(
f\right) }{\Delta _{h}^{\left( m,p\right) }\left( g\right) }\right] ^{\frac{1%
}{\rho _{h}^{\left( m,p\right) }\left( g\right) }}\right\} \\
\leq \max \left\{ \left[ \frac{\overline{\Delta }_{h}^{\left( m,q\right)
}\left( f\right) }{\overline{\Delta }_{h}^{\left( m,p\right) }\left(
g\right) }\right] ^{\frac{1}{\rho _{h}^{\left( m,p\right) }\left( g\right) }%
},\left[ \frac{\Delta _{h}^{\left( m,q\right) }\left( f\right) }{\Delta
_{h}^{\left( m,p\right) }\left( g\right) }\right] ^{\frac{1}{\rho
_{h}^{\left( m,p\right) }\left( g\right) }}\right\} \leq \overline{\tau }%
_{g}^{\left( p,q\right) }\left( f\right) \leq \left[ \frac{\Delta
_{h}^{\left( m,q\right) }\left( f\right) }{\overline{\Delta }_{h}^{\left(
m,p\right) }\left( g\right) }\right] ^{\frac{1}{\rho _{h}^{\left( m,p\right)
}\left( g\right) }}
\end{multline*}%
and 
\begin{multline*}
\left[ \frac{\tau _{h}^{\left( m,q\right) }\left( f\right) }{\overline{\tau }%
_{h}^{\left( m,p\right) }\left( g\right) }\right] ^{\frac{1}{\lambda
_{h}^{\left( m,p\right) }\left( g\right) }}\leq \overline{\Delta }%
_{g}^{\left( p,q\right) }\left( f\right) \leq \min \left\{ \left[ \frac{\tau
_{h}^{\left( m,q\right) }\left( f\right) }{\tau _{h}^{\left( m,p\right)
}\left( g\right) }\right] ^{\frac{1}{\lambda _{h}^{\left( m,p\right) }\left(
g\right) }},\left[ \frac{\overline{\tau }_{h}^{\left( m,q\right) }\left(
f\right) }{\overline{\tau }_{h}^{\left( m,p\right) }\left( g\right) }\right]
^{\frac{1}{\lambda _{h}^{\left( m,p\right) }\left( g\right) }}\right\} \\
\leq \max \left\{ \left[ \frac{\tau _{h}^{\left( m,q\right) }\left( f\right) 
}{\tau _{h}^{\left( m,p\right) }\left( g\right) }\right] ^{\frac{1}{\lambda
_{h}^{\left( m,p\right) }\left( g\right) }},\left[ \frac{\overline{\tau }%
_{h}^{\left( m,q\right) }\left( f\right) }{\overline{\tau }_{h}^{\left(
m,p\right) }\left( g\right) }\right] ^{\frac{1}{\lambda _{h}^{\left(
m,p\right) }\left( g\right) }}\right\} \leq \Delta _{g}^{\left( p,q\right)
}\left( f\right) \leq \left[ \frac{\overline{\tau }_{h}^{\left( m,q\right)
}\left( f\right) }{\tau _{h}^{\left( m,p\right) }\left( g\right) }\right] ^{%
\frac{1}{\lambda _{h}^{\left( m,p\right) }\left( g\right) }}~.
\end{multline*}
\end{theorem}

\end{document}